\documentclass[sn-mathphys]{sn-jnl}
\usepackage[utf8]{inputenc}
\usepackage{graphicx}
\usepackage{subcaption}
\usepackage{amsthm,amsmath,amsfonts,amssymb}
\usepackage{dsfont}
\usepackage{xcolor}
\RequirePackage[numbers,sort&compress]{natbib}%

\usepackage{hyperref}
\hypersetup{colorlinks=true, linkcolor=blue, urlcolor=blue, citecolor=gray}

\theoremstyle{plain}
\newtheorem{theo}{Theorem}
\newtheorem{lemm}{Lemma}
\newtheorem{coro}{Corollary}
\newtheorem{prop}{Proposition}

\theoremstyle{definition}
\newtheorem{exam}{Example}
\newtheorem{defi}{Definition}
\newtheorem{assu}{Assumption}
\newtheorem{rema}{Remark}

\renewcommand{\hat}{\widehat}
\renewcommand{\tilde}{\widetilde}

\newcommand{\E}{{\mathbb E}}
\renewcommand{\R}{{\mathbb R}}
\newcommand{\N}{\mathbb N}

\newcommand{\bigo}{\mathcal{O}}
\newcommand{\littleo}{o}

\newcommand{\cl}{\mathrm{cl}}
\newcommand{\reach}{\mathrm{reach}}
\newcommand{\rconv}{\mathrm{rconv}}
\newcommand{\rch}[2]{\hat{\mathrm{rch}}^{(\epsilon_{#2})}(\hat{#1}^{(#2)})}
\newcommand{\rnv}[2]{\hat{\mathrm{r}}^{(\epsilon_{#2})}(\hat{#1}^{(#2)})}

\newcommand{\norm}[1]{\left\| #1 \right\|_2}

\newcommand{\pc}{\Phi} % point cloud

\definecolor{R}{RGB}{255, 150, 0}
\definecolor{T}{RGB}{0, 100, 255}

\begin{document}

\title[Computable bounds for the reach and $r$-convexity of subsets of $\R^d$]{Computable bounds for the reach and $r$-convexity of subsets of $\R^d$}
\author{\centering Ryan Cotsakis\\
\vspace{12pt}
{\small
Laboratoire J.A. Dieudonné, Université Côte d'Azur\\
28, Avenue Valrose, 06108 Nice Cedex 2, France}\\
\vspace{12pt}
{\small ryan.cotsakis@univ-cotedazur.fr}}

\abstract{The convexity of a set can be generalized to the two weaker notions of reach and $r$-convexity; both describe the regularity of a set's boundary. For any compact subset of $\R^d$, we provide methods for computing upper bounds on these quantities from point cloud data. The bounds converge to the respective quantities as the point cloud becomes dense in the set, and the rate of convergence for the bound on the reach is given under a weak regularity condition. We also introduce the $\beta$-reach, a generalization of the reach that excludes small-scale features of size less than a parameter $\beta\in[0,\infty)$. Numerical studies suggest how the $\beta$-reach can be used in high-dimension to infer the reach and other geometric properties of smooth submanifolds. }


\maketitle

{\small \textbf{Acknowledgements}}
\vspace{5pt}

{\small I extend my gratitude to Elena Di Bernardino and Thomas Opitz for their attentive and indispensable guidance. Thank you to the anonymous reviewers, who significantly helped to improve the communication of ideas in this manuscript. This work has been supported by the French government, through the 3IA C\^{o}te d'Azur Investments in the Future project managed by the National Research Agency (ANR) with the reference number ANR-19-P3IA-0002.}
\newpage

\section{Introduction}
A number of concepts from convex geometry generalize from convex sets to much larger classes of sets. A classic example from \cite{federer1959} is the extension of kinematic formulas (in particular, Steiner's formula) for convex sets to sets with \emph{positive reach} (see Definition~\ref{def:reach}). Another example is the notion of the convex hull of a set, which can be weakened to the \textit{$r$-convex hull}, for $r > 0$. This weak notion of a convex hull, instead of being expressed as the intersection of half-spaces, is expressed in terms of intersections of the complements of open balls of radius $r$.
The resulting intersection is said to be \emph{$r$-convex} (see Definition~\ref{def:r-convex}), which differs subtly from the notion of reach, and these differences have been studied in \cite{colesanti2010} and \cite{cuevas2012} for example. Both the reach and $r$-convexity---ranging from 0 to $\infty$ inclusive---can be seen as measures of the degree to which a set is convex. This paper introduces methods of computing upper bounds for these quantities from point cloud data in several, general settings.
\smallskip

There is a vast literature on properties of $r$-convex sets, dating from \cite{perkal1956}. Their relations to other classes of sets that generalize convexity have been studied in \cite{serra1984, walther1999, colesanti2010, cuevas2012}. The use of $r$-convexity in image smoothing has been suggested in \cite{serra1984, walther1999, cuevas2007} as well as for set estimation in \cite{manilevitska1993, cuevas2007, pateirolopez2008, cuevas2009, cuevas2012, aaron2022}. In other literature, the $r$-convex hull of a set is referred to as its \textit{double offset} \citep{chazal2007, chazal2009_1, kim2019}. In \cite{cuevas2004}, the authors use a rolling-type condition (weaker than $r$-convexity) to improve the rate of convergence of their set boundary estimator. Efforts towards estimating the largest $r$ such that a set is  $r$-convex have been made in \cite{rodriguezcasal2016}; the authors test the uniformity of a point cloud on its $r$-convex hull using the statistical test of uniformity proposed in \cite{berrendero2012}.
\smallskip

%\rya{Tthe reach of $A$ is also referred to as the condition number of $A$ (see, \emph{e.g.}, \cite{niyogi2008})}
The reach is a popular measure of convexity largely due to its link with the Steiner formula; for the larger the reach of a set, the larger the interval on which the volume of a dilation of the set is polynomial in the dilation radius \cite[Theorem~5.6]{federer1959}. The positivity of the reach implies a number of desirable properties, and so positive reach is a common hypothesis to guarantee convergence rates of statistical estimators of geometric quantities \citep{chazal2008, thale2008, cuevas2009, rataj2019, bierme2019, cotsakis2022}. In the field of geometric data analysis, the reach constitutes a commonly used measure of regularity of a set's boundary, and is hence also referred to as the \textit{condition number} \cite{niyogi2008}. In topological data analysis, a set's reach is shown to quantify the ability to infer its homology from point cloud data
\citep{kim2019, niyogi2008, lieutier2004, chazal2005_1}.

The statistical estimation of the reach has been of particular interest in recent literature. \cite{aamari2019} and \cite{aamari2019_1} suggest estimating the reach of a smooth submanifold of $\R^d$ by measuring distances between it and its tangent spaces---a strategy inspired by the formulation of the reach in \cite[Theorem~4.18]{federer1959}. These works obtain  bounds on the minimax rate of convegence for this estimator when the manifold is $C^3$-smooth, and when its tangent spaces are known. Following up to these works, \cite{aamari2022} establishes an optimal convergence rate for minimax estimators of the reach of $C^k$-smooth submanifolds of $\R^d$ without boundary. The rates that they establish adapt to the smoothness of the manifold, and to the type of phenomenon that limits the reach (see Remark~\ref{rem:aamari_split}).
\cite{berenfeld2022} looks to the \textit{convexity defect function} introduced in \cite{attali2013} and suggests an estimator for the reach of $C^k$-smooth manifolds by establishing a link between the convexity defect function and the reach. Convergence rates of their estimator are given for $k\geq 3$. In \cite{cholaquidis2022}, the authors introduce a complete and tractable method for estimating the reach from point cloud data involving the computation of graph distances in a spatial network defined over the point cloud. Their approach is based on the formulation of the reach introduced in \cite[Theorem~1]{boissonnat2019}.

Other strategies for estimating the reach involve a prior estimation of the \textit{medial axis} (see Definition~\ref{def:medial_axis}) \citep{dey2003, dey2006}. The $\lambda$-medial axis \citep{chazal2005}, $\mu$-medial axis \citep{chazal2009}, and $(\lambda,\alpha)$-medial axis \citep{lieutier2023} are all generalizations of the medial axis, and the reach can conceivably be approximated by the minimal distance from a set to the estimate of the medial axis of its complement. Such a strategy is suggested in \cite{cuevas2014} for the $\lambda$-medial axis, where the authors heavily rely on the notion of $r$-convexity in their construction. 

A number of other authors have taken an interest in the mathematical properties of the reach. \cite{poliquin2000} identifies the sets of reach $r$ with the $r$-proximally smooth sets introduced in \cite{clarke1995}, implying that the reach can be characterized by the gradients of the distance-to-set function. \cite{colesanti2010} makes a number of insightful connections between the reach and other geometrical properties of sets. \cite{attali2015} studies Vietoris–Rips complexes via the reach, proving that the reach of a set can only increase if intersected by sufficiently small balls. An alternative characterization of the reach, involving pairwise geodesic distances, is provided in \cite{boissonnat2019}.
The authors also study the relationship between midpoints of pairs of points in a set, and the set's reach \cite[Lemma~1]{boissonnat2019}. We base the construction of our bound for the reach on this result (see Theorem~\ref{thm:equivalent_reach}).
\smallskip

This paper makes steps towards providing computationally tractable methods for bounding the reach and $r$-convexity of subsets of $\R^d$ given point cloud data that represent the underlying sets.

Firstly, we establish some facts about the reach of closed subsets of $\R^d$. We prove that the $r$-convexity and reach are equivalent for compact subsets of $\R^d$ whose topological boundary is a $C^1$-smooth, $(d-1)$-dimensional manifold without boundary (see Theorem~\ref{thm:reach_closing}). In addition, for closed subsets of $\R^d$, we introduce the $\beta$-reach (see Definition~\ref{def:beta_reach}), a quantity that loosely represents the reach of a set when features of size less than $\beta\in[0,\infty)$ are ignored. Indeed, the $\beta$-reach is identified with the reach for $\beta=0$ (see Theorem~\ref{thm:equivalent_reach}).

These ideas are used to create methods of inferring bounds on the reach and $r$-convexity of sets from point cloud data. For general, closed subsets of $\R^d$ (possibly having finite $d$-volume, and having \textit{no smoothness conditions} on their topological boundaries), we provide an upper bound of the $r$-convexity of the set based on samples of the set and its complement at sampling locations that extend over $\R^d$. We show that, as the sampling locations become dense in $\R^d$, this bound converges to the largest $r$ such that the set is $r$-convex (see Theorem~\ref{thm:rconv_estimator}). An example on real data in 3 dimensions shows that this method identifies regions where the underlying set is not locally $r$-convex, for $r>0$, with a test specificity of 100\%.
Similarly, for any closed set $A\subseteq\R^d$, we define an upper bound on the reach based on a set of points known to reside in $A$. As the set of points converges in the Hausdorff metric to $A$, we show that the bound converges to the reach of $A$, and provide the rate of convergence in terms of the Hausdorff distance between $A$ and the sample points (see Theorem~\ref{thm:reach_estimator}). A weak regularity condition on the $\beta$-reach of $A$, for $\beta$ near 0, is used to show the convergence of our upper bound with a rate. In practice, both the $\beta$-reach of a point cloud and the upper bound on the reach can be computed efficiently in high-dimension. The computational complexity of the method increases only linearly with the dimension of the ambient space $d$.
\smallskip

The organization of the paper is as follows. Section~\ref{sec:definitions} introduces the notation that we use throughout the document, explores the relationships between the reach and $r$-convexity, and introduces the $\beta$-reach.
Section~\ref{sec:point_cloud} describes the three methods that we propose for inferring bounds and approximations for the reach and $r$-convexity of general compact subsets of $\R^d$ from point cloud data. The bounds for the $r$-convexity and reach are studied in Sections~\ref{sec:point_cloud_rconv} and~\ref{sec:point_cloud_reach} respectively. Section~\ref{sec:point_cloud_beta_reach} elaboates on how the $\beta$-reach of point clouds can be used to approximate the $\beta$-reach of the sets that they represent. In Section~\ref{sec:sim}, we provide numerical studies that underline the computability of our methods. An application of the methods on real data is given in Section~\ref{sec:sim:aircraft}. In Section~\ref{sec:sim_convergence}, the methods are tested against simulated data for which the reach and $r$-convexity are known, and empirical rates of convergence are provided.
Some technical proofs and auxiliary results are postponed to Section~\ref{sec:proofs}.

\section{Definitions and important notions}\label{sec:definitions}

The sets that we study in this paper are subsets of $\R^d$, endowed with the Euclidean metric $\norm{\cdot}$. For a set $S\subseteq\R^d$, let $\partial S$ denote its topological boundary, let $\cl(S) := S\cup \partial S$ denote its closure, and let $S^c$ denote its complement in $\R^d$. Denote the closed ball with radius $r\in\R^+$ centered at $s\in\R^d$ by $B(s,r):=\{t\in\R^d : \norm{t-s} \leq r\}$. For $t\in\R^d$, denote the distance between $t$ and a non-empty set $S$ by $\delta_S(t):= \inf\{\norm{t-s} : s\in S\}$.

\subsection{Set dilation, set erosion, and combinations of the two}

\begin{defi}[Operations on subsets of $\R^d$]\label{def:minkowski}
We recall the Minkowski addition of two sets $A, B \subseteq \R^d$,
$$A \oplus B := \{x+y:x\in A, y\in B\}.$$
The Minkowski difference is given by
$$A \ominus B := \bigcap_{y\in B}(A\oplus \{y\}) = (A^c \oplus B)^c,$$
where $A^c$ denotes the complement of $A$ in $\R^d$. For $r\in\R$, let
$$A_r := \begin{cases}
        A\oplus B(0,r), & \text{for } r\geq 0,\\
        A\ominus B(0,-r), & \text{for } r < 0,
        \end{cases} 
$$
denote either the dilation or erosion of a set $A$, depending on the sign of $r$. Finally, define $A_{\bullet r}:= (A_r)_{-r}$ to be the closing of $A$ by $B(0,r)$ if $r\geq 0$ and the opening of $A$ by $B(0,-r)$ if $r<0$.
\end{defi}

For $r\geq 0$ and $A\subseteq \R^d$ closed, it follows from Definition~\ref{def:minkowski} that $A_r$ (also known as the \textit{$r$-offset} of $A$ in, \textit{e.g.}, \cite{chazal2007}) denotes all the points in $\R^d$ within a distance $r$ of the set $A$. The set $A_{-r}$ denotes all the points in $A$ a distance of at least $r$ from $\partial A$.
\smallskip

With these notions established, the Hausdorff distance between two closed sets $A,B\subseteq \R^d$ is defined as $d_H(A,B):= \inf\{r \in\R^+ : A \subseteq B_r, B \subseteq A_r\}$.

\begin{lemm}\label{lem:additivity_of_dilation}
Let $r,s > 0$, and let $A\subseteq \R^d$. The following identities hold:
\begin{enumerate}
\item[(a)] $\qquad(A_r)^c = (A^c)_{-r}$,
\item[(b)] $\qquad A\subseteq A_{\bullet r}$,
\item[(c)] $\qquad(A_r)_s = A_{r+s}$,
\item[(d)] $\qquad(A_r)_{-s} \supseteq A_{r-s}$,
\item[(e)] $\qquad(A_{-r})_s \subseteq A_{s-r}$.
\end{enumerate}
\end{lemm}
\begin{proof}[Proof of Lemma~\ref{lem:additivity_of_dilation}]
Fix $r,s > 0$. The identity in~(a) follows directly from Definition~\ref{def:minkowski}. Item~(b) is proved by contradiction. Let $a\in A$ and suppose $a\in (A_{\bullet r})^c = (A_r)^c \oplus B(0,r)$. Then there is a $p\in (A_r)^c$ such that $a\in B(p,r) \Leftrightarrow \norm{p-a} \leq r \Leftrightarrow p \in B(a,r) \Leftrightarrow p\in A_r$; thus, a contradiction. To prove~(c), remark that $B(0,r) \oplus B(0,s) = B(0,r+s)$ and that Minkowski addition is associative. To prove~(d), consider the case where $r \geq s$, then $(A_r)_{-s} = (A_{r -s + s})_{-s} = ((A_{r -s})_s)_{-s} = (A_{r -s})_{\bullet s} \supseteq A_{r -s}$. For the case $r < s$, write $(A_r)_{-s} = (A_r)_{r-s-r} = ((A_r)_{-r})_{r-s} = (A_{\bullet r})_{r-s} \supseteq A_{r -s}$. To show~(e), consider the complements of the sets in~(d) and apply~(a) repeatedly.
\end{proof}

\subsection{The reach and related concepts}

\begin{defi}[The reach]\label{def:reach}
Recall from \cite{federer1959} that the \emph{reach} of a set $A\subseteq \R^d$ is given by
\begin{equation*}
\reach(A) := \sup\big\{r \in \R^+: \forall y \in A_r\ \exists!x\in A\ \mathrm{nearest\ to}\ y\big \}.
\end{equation*}
If $\reach(A) >0$, then $A$ is said to have \emph{positive reach}. 
\end{defi}

A useful notion related to the reach of a closed set $A$ is the \textit{medial axis} of $A^c$, originally proposed in \cite{blum1967}.
\begin{defi}[The medial axis]\label{def:medial_axis}
Let $O\subseteq \R^d$ be open. Its \textit{medial axis} $\mathcal{M}(O)$ is the set of points in $O$ with at least two closest points in $\partial O$.
\end{defi}
The reach of a closed set $A$ can be alternatively expressed as
\begin{equation}\label{eqn:medial_axis_reach}
\reach(A) = \inf\{\norm{a-x}: a\in A,\ x\in \mathcal{M}(A^c)\}.
\end{equation}

\subsubsection{Connections to $r$-convexity}

\begin{defi}[$r$-convexity]\label{def:r-convex}
A set $A\subseteq \R^d$ is said to be \emph{$r$-convex} for $r\in \R^+$ if it is closed and $A_{\bullet s} = A$ for all $s\in (0,r)$ (see, \emph{e.g.}, \cite{perkal1956}). Define the quantity $\rconv(A) := \sup\{r\in\R : A_{\bullet r} = A\}$.
\end{defi}

An equivalent definition of $r$-convexity is as follows: a set $A\subseteq \R^d$ is $r$-convex if and only if it can be expressed as the complement of a union of open balls of radius $r$.
%This holds since $A_{\bullet r} = r$ if and only if $A$ can be expressed as the complement of a union of closed balls of radius $r$.
%And it follows from this notion that an $r$-convex set is $s$-convex for all $s\in [0,r]$.

\begin{theo}\label{thm:reach_closing}
Let $A$ be closed in $\R^d$. It holds that 
\begin{equation}\label{eqn:reach_inclusion}
%\rya{-\rconv\big(\cl(A^c)\big) \leq -\reach(\cl(A^c)) \leq }\reach(A) \leq \rconv(A).
\reach(A) \leq \rconv(A).
%\big(-\reach(\cl(A^c)),\ \reach(A)\big) \subseteq \{r\in\R:A_{\bullet r} = A\}.
\end{equation}
Moreover, if $\partial A$ is a $C^1$-smooth $(d-1)$-dimensional manifold without boundary, then
\begin{equation}\label{eqn:reach_equality}
\reach(A) = \rconv(A).
%\big[-\reach(\cl(A^c)),\ \reach(A)\big] = \cl\big(\{r\in\R:A_{\bullet r} = A\}\big).
\end{equation}
\end{theo}

Equation~\eqref{eqn:reach_inclusion} is proven in \citet[Proposition~1]{cuevas2012} for compact $A$. Nonetheless, we reprove the statement for closed $A$ in the proof of Theorem~\ref{thm:reach_closing}, which we postpone to Section~\ref{sec:proofs}. The novelty in Theorem~\ref{thm:reach_closing} is that it provides a class of subsets of $\R^d$ for which the reach and $r$-convexity are equal.
 
To see that the $r$-convexity and the reach of a set are indeed distinct notions for general subsets of $\R^d$, Figure~\ref{fig:reach} provides an example of a closed set $A$ for which $\reach(A) < \rconv(A)$. Consider also the following remark.

\begin{figure}[t]
\centering
\includegraphics[width=0.75\linewidth]{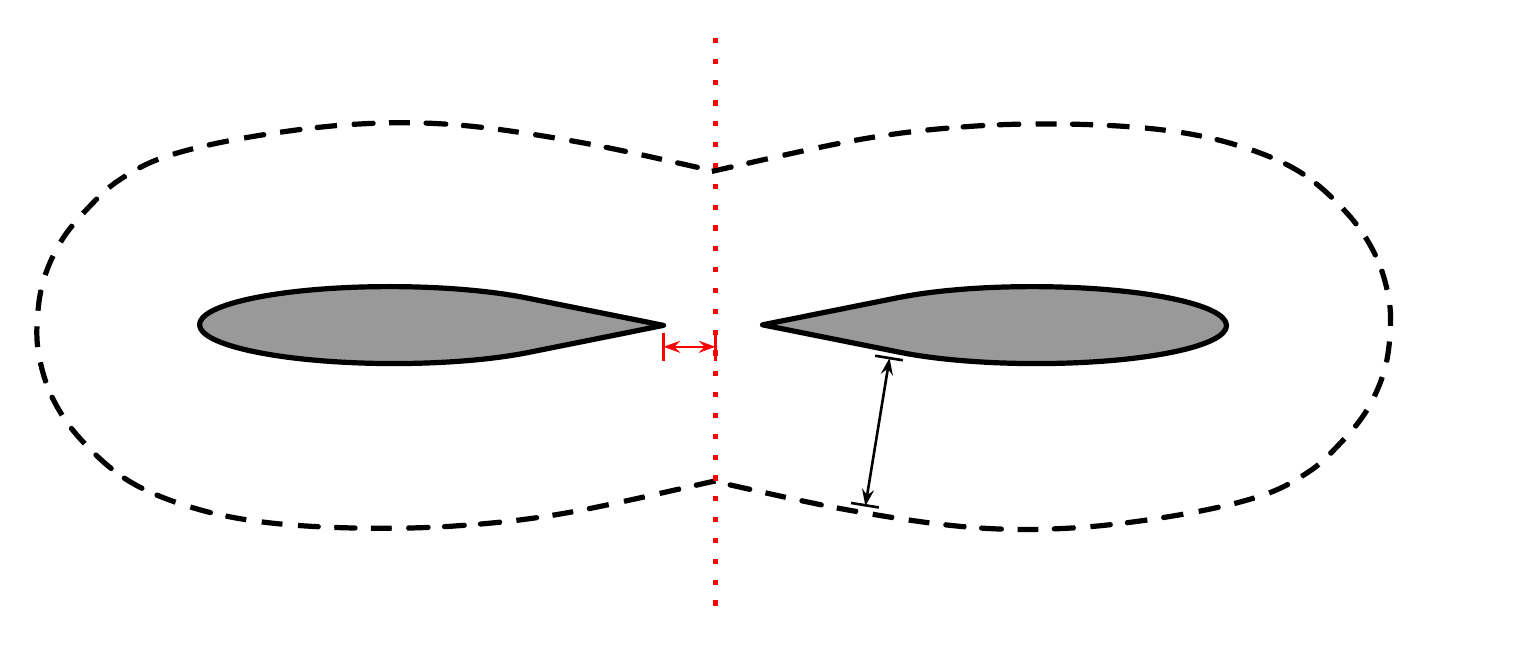}
\put(-171,40){\small{\color{red}{$\reach(A)$}}}
\put(-105,36){\small{$\rconv(A)$}}
\caption{The closed set $A$ (in grey) has $\rconv(A)>\reach(A)$. The set $A$ can be expressed as the complement of a union of open balls of radius $\rconv(A)$ whose centers lie outside the dotted black curve. The dotted red line is the medial axis $\mathcal{M}(A^c)$.}
\label{fig:reach}
\end{figure}

\begin{rema}
Any closed subset of $\R^d$ contained in a $(d-1)$-dimensional affine linear subspace is $r$-convex for all $r > 0$. This is easy to see since the complement of the set in the $(d-1)$-dimensional affine linear subspace is open, and is the union of open $(d-1)$-balls of radius less than $r$. Each $(d-1)$-ball can be expressed as the intersection of the affine linear subspace with an open ball of radius $r$ in $\R^d$; therefore, the closed subset in question can be expressed as the complement of a union of open balls of radius $r$ and is hence $r$-convex.

Recently, a counterexample to Borsuk's conjecture that $r$-convex sets are locally contractible, was published in \cite{cholaquidis2023}. The conjecture is easily seen to be false by mapping a closed set in $\R^{d-1}$ that is not locally contractible to a $(d-1)$-dimensional hyperplane in $\R^d$ under an isometry.
\end{rema}

Here, we provide some corollaries to Theorem~\ref{thm:reach_closing} that provide alternative sets of sufficient conditions for the equality of the reach and the $r$-convexity. The first applies to sets in Serra's regular model \cite[p.~144]{serra1984}.

\begin{coro}\label{cor:serra}
Let $A\subset\R^d$ be non-empty, compact, and path-connected. If $\rconv(\cl(A^c)) > 0$, then $\reach(A) = \rconv(A)$.
\end{coro}

The proof of Corollary~\ref{cor:serra}, which relies heavily on Theorem~1 in \cite{walther1999}, is postponed to Section~\ref{sec:proofs}. Likewise, we prove the following result in Section~\ref{sec:proofs}.

\begin{coro}\label{cor:small_dilation}
Let $A$ be a closed set in $\R^d$ and let $\epsilon \in\R^+$. Then $\reach(A_\epsilon) = \rconv(A_\epsilon)$.
\end{coro}

\subsubsection{The $\beta$-reach}\label{sec:definitions_reach_beta_reach}

We show in Theorem~\ref{thm:equivalent_reach} below, that the reach of a set can be formulated in terms of pairs of points in the set, and the distance to the set from their midpoints. We define a parametrized version of the reach by restricting to pairs of points whose midpoints are sufficiently far from the set. The so-called $\beta$-reach is constructed from the following family of functions.

\begin{defi}[Spherical cap geometry]\label{def:g_h_b}
Define for $\alpha\in[0,\infty)$ and $x\in [0,\alpha/2]$,
\begin{equation}\label{eqn:g_ell}
g_\alpha(x) := \begin{cases} \frac{\alpha^2}{8x} + \frac{x}2, &\qquad x > 0,\\
\infty, &\qquad x = 0,
\end{cases}
\end{equation}
and its inverse for $r \geq \alpha/2$,
\begin{equation}\label{eqn:g_ell_inv}
g_\alpha^{-1}(r) := r - \sqrt{r^2 - \frac{\alpha^2}4}.
\end{equation}
\end{defi}

\begin{figure}[t]
\centering
\includegraphics[width=0.5\linewidth]{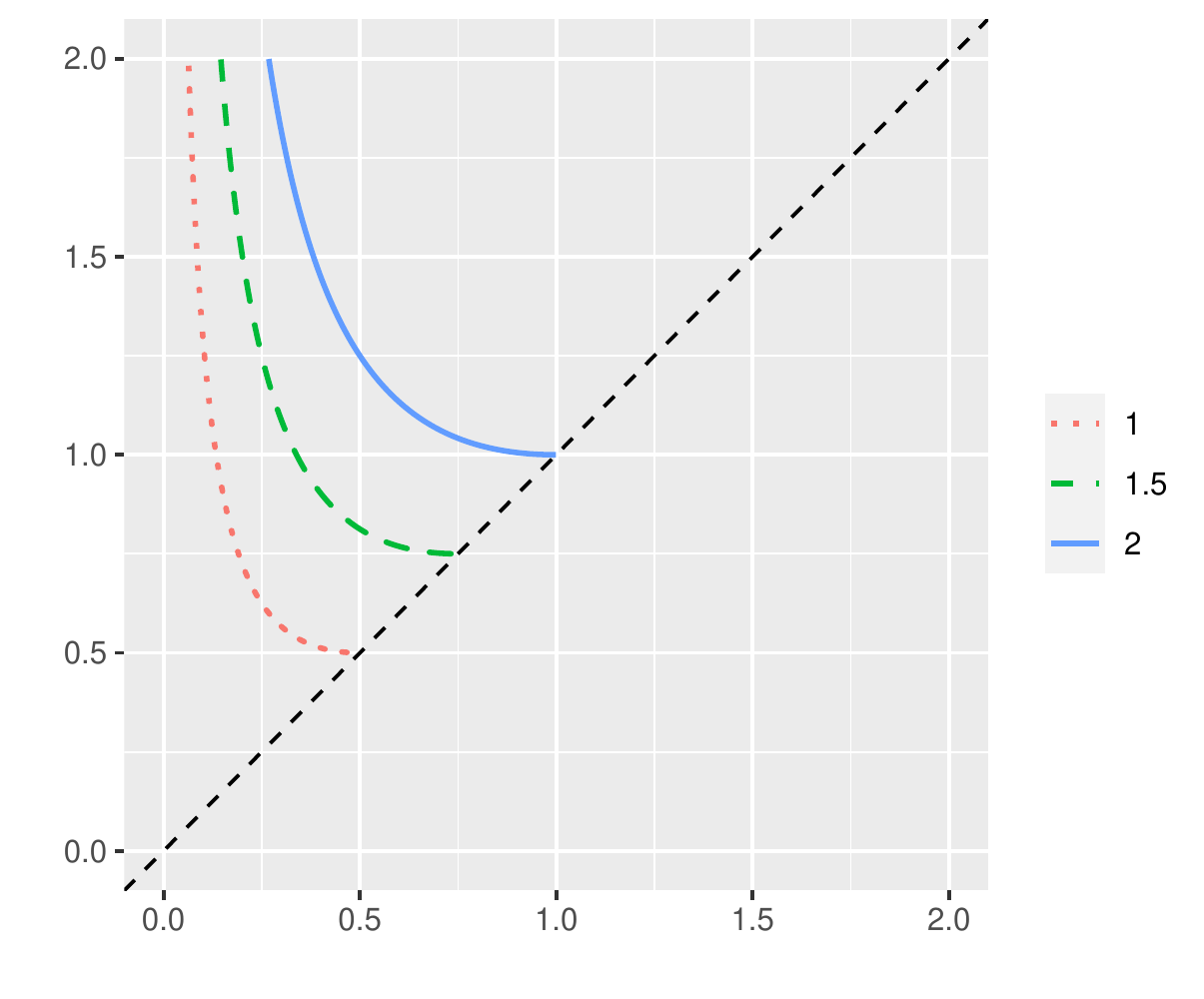}
\put(-20,90){$\alpha$}
\put(-90,0){$x$}
\put(-190,110){$g_\alpha(x)$}
\caption{$g_\alpha(x)$ in Definition~\ref{def:g_h_b} plotted for several values of $\alpha$.}
\label{fig:g_alpha_curves}
\end{figure}

\begin{figure}[t]
\centering
\includegraphics[width=0.4\linewidth]{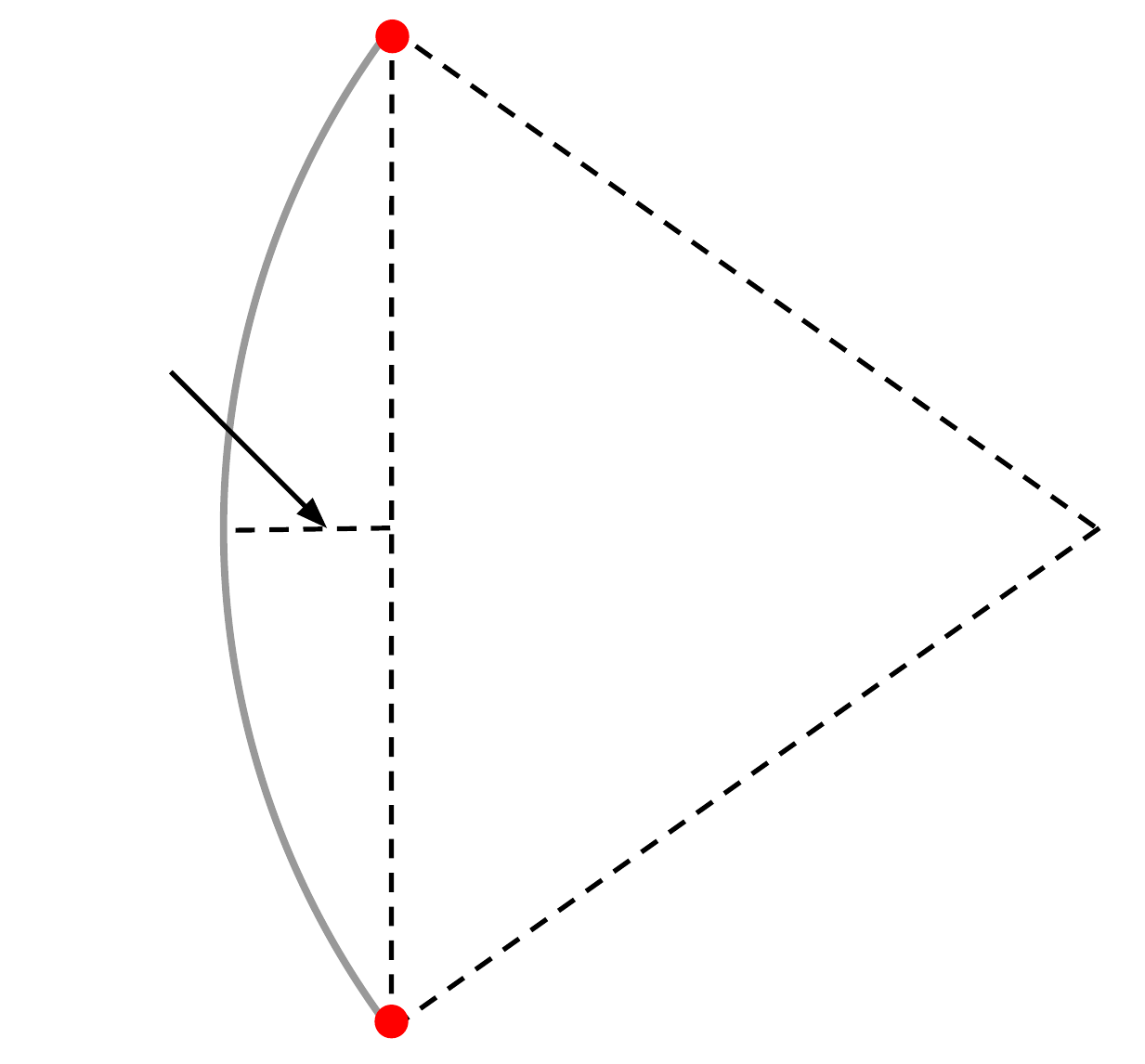}
\put(-45,105){$r$}
\put(-100,130){$a_1$}
\put(-100,-4){$a_2$}
\put(-163,93){$g_{\norm{a_2-a_1}}^{-1}(r)$}
\caption{A geometric interpretation of $g_\alpha(x)$ in Definition~\ref{def:g_h_b}.}
\label{fig:g_alpha}
\end{figure}
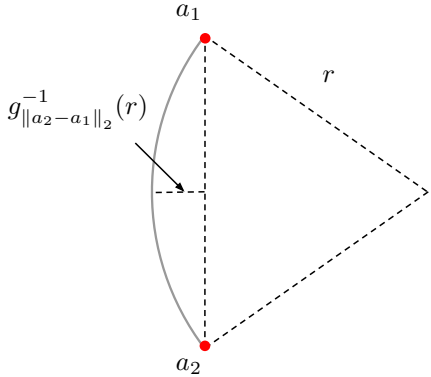

See Figure~\ref{fig:g_alpha} for a geometric interpretation of the function in~\eqref{eqn:g_ell} and its inverse in~\eqref{eqn:g_ell_inv}. Evidently from the figure, this function is derived from the height of a spherical cap; it is written about in \cite{attali2013, attali2015, berenfeld2022, boissonnat2019, divol2021} in reference to the reach.

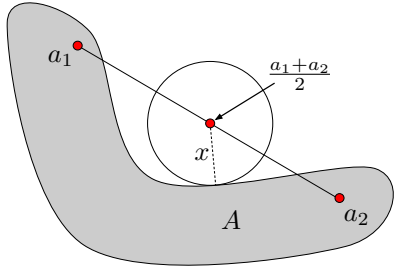
\begin{figure}[t]
\centering
\includegraphics[width=0.45\linewidth]{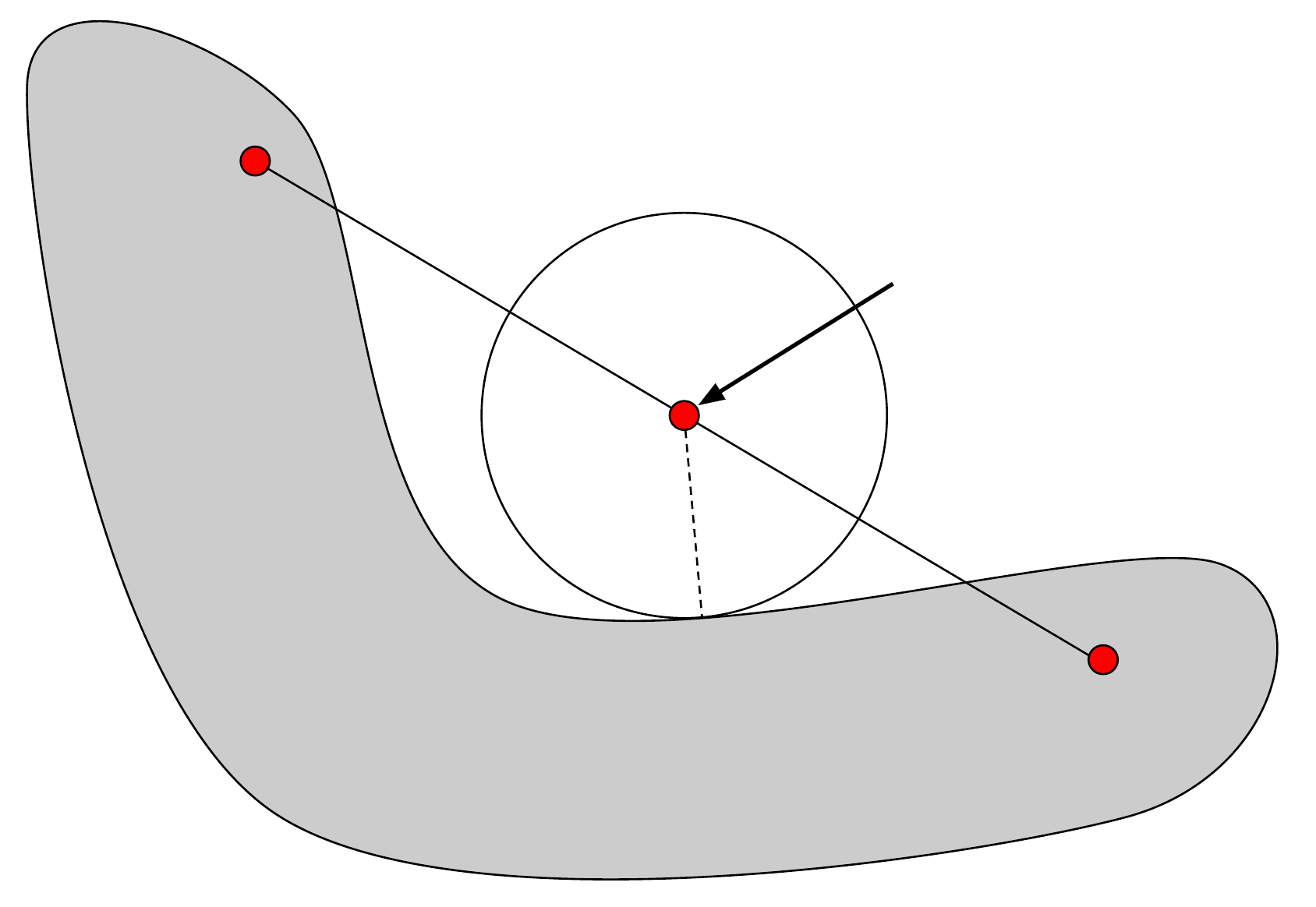}
\put(-79,43){$x$}
\put(-69,17){$A$}
\put(-134,80){$a_1$}
\put(-52,73){$\frac{a_1+a_2}2$}
\put(-23,20){$a_2$}
\caption{For $\beta\in[0,\infty)$, the $\beta$-reach of $A$ (see Definition~\ref{def:beta_reach}) is the largest lower bound of $g_{\norm{a_2-a_1}}(x)$ (see Definition~\ref{def:g_h_b}), for the pairs $a_1,a_2\in A$ satisfying $x\geq \beta$.}
\label{fig:beta_reach_def}
\end{figure}

\begin{defi}[The $\beta$-reach]\label{def:beta_reach}
For a closed set $A\subseteq\R^d$ and $\beta\in [0,\infty)$, let the \emph{$\beta$-reach} of $A$ be defined as
\begin{equation*}
\reach_\beta(A) := \inf\left\{g_{\norm{a_2-a_1}}(x) : a_1,a_2 \in A,\ x=\delta_A\left(\frac{a_1+a_2}2\right) \geq \beta\right\},
\end{equation*}
where $g_{\norm{a_2-a_1}}(x)$ is defined in~\eqref{eqn:g_ell}. Recall that $\delta_A:\R^d\to \R$ maps each point in $\R^d$ to its distance from $A$. See Figure~\ref{fig:beta_reach_def} for a visual aid.
\end{defi}

By restricting to pairs of points $a_1,a_2\in A$ in Figure~\ref{fig:beta_reach_def} that yield $x\geq \beta$, the $\beta$-reach of $A$ is the largest lower bound of the resulting values of $g_{\norm{a_2-a_1}}(x)$. If one does not restrict the size of $x$ (\textit{i.e.}, for $\beta=0$), then this largest lower bound is precisely $\reach(A)$. This is formalized by the following theorem.

\begin{theo}\label{thm:equivalent_reach}
Let $A$ be closed in $\R^d$. The map $\beta \mapsto \reach_\beta(A)$ for $\beta\in \R^+$ is non-decreasing; moreover,
\begin{equation}\label{eqn:formulation_set_distance}
\lim_{\beta\searrow 0}\reach_{\beta}(A) = \reach_0(A) = \reach(A).
\end{equation}
\end{theo}

\begin{proof}[Proof of Theorem~\ref{thm:equivalent_reach}]
The non-decreasing property is seen immediately via the inclusion
$$\left\{(a_1,a_2)\in A^2 : \delta_A\left(\frac{a_1+a_2}2\right) \geq \beta_2\right\} \subseteq \left\{(a_1,a_2)\in A^2 : \delta_A\left(\frac{a_1+a_2}2\right) \geq \beta_1\right\},$$
for all $\beta_1,\beta_2\in\R$ satisfying $\beta_1 < \beta_2$.

\begin{figure}[t]
\centering
\includegraphics[width=0.55\linewidth]{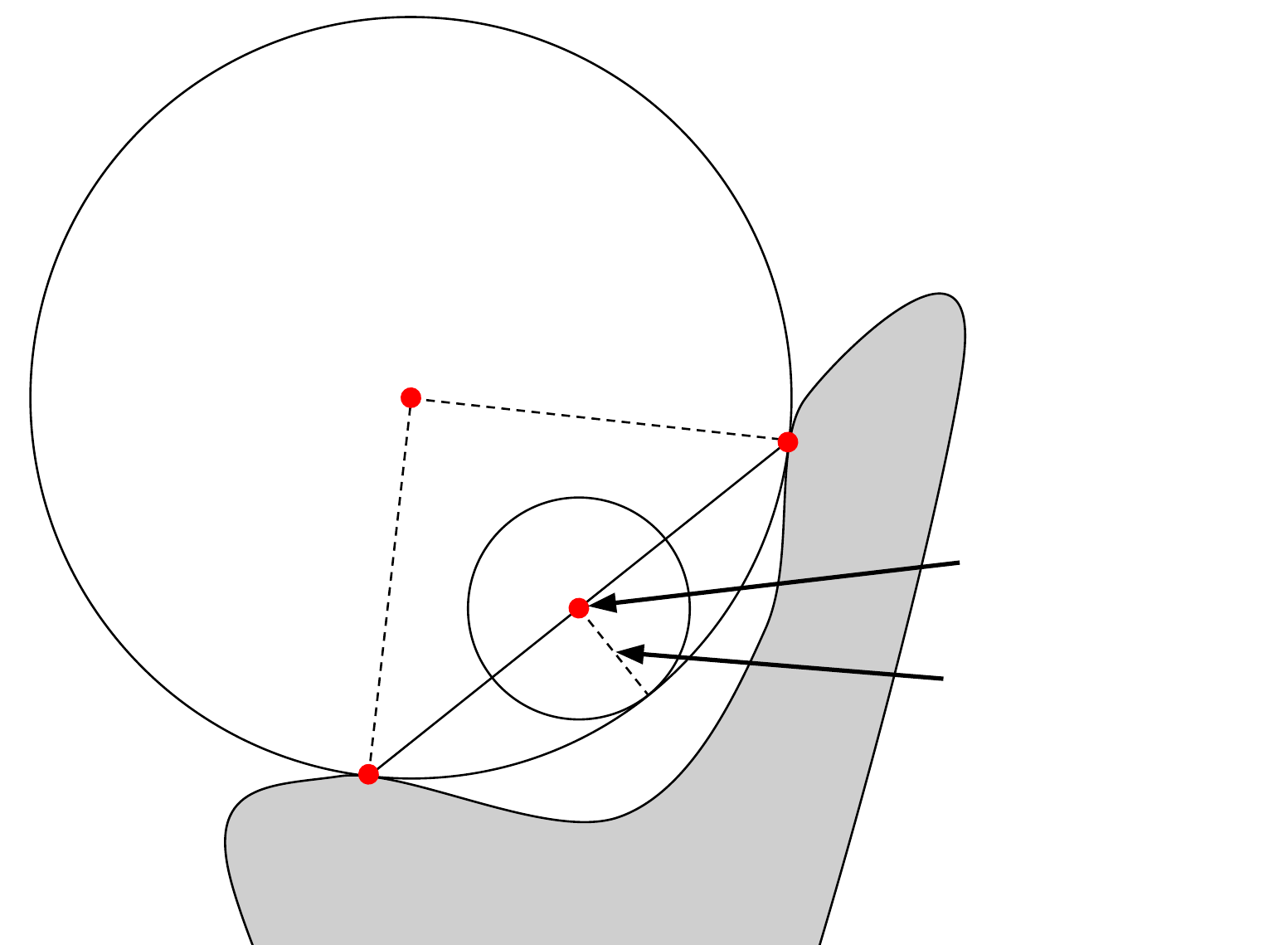}
\put(-136,80){$p$}
\put(-83,10){$A$}
\put(-137,13){$a_1$}
\put(-42,55){$\frac{a_1+a_2}2$}
\put(-69,70){$a_2$}
\put(-43,32){$g_{\norm{a_2-a_1}}^{-1}(\delta_A(p))$}
\caption{The interior of the large ball centered at $p$ does not intersect $A$, so the distance from the midpoint $(a_1 +a_2)/2$ to the set $A$ is at least $g_{\norm{a_2-a_1}}^{-1}(\delta_A(p))$. This construction is used in the proof of Theorem~\ref{thm:equivalent_reach}.}
\label{fig:proof_theorem_2}
\end{figure}
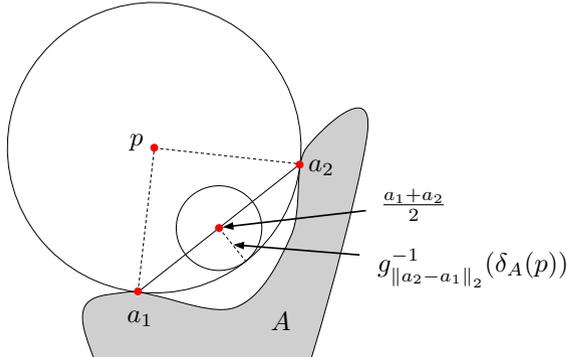

Now, we start by proving the second equality in~\eqref{eqn:formulation_set_distance}.
The proof is largely supplied by Lemma~1 of \cite{boissonnat2019} which provides
\begin{equation}\label{eqn:boissonnat_lemma_1}
\reach(A) \leq g_{\norm{a_2-a_1}}\circ\delta_A\left(\frac{a_1+a_2}2\right),
\end{equation}
for all $a_1, a_2 \in A$, since $g_\alpha(x)$ is non-increasing in $x$. What remains to show is that $\reach(A)$ is the largest lower bound in \eqref{eqn:boissonnat_lemma_1}; \textit{i.e.}, if $r > \reach(A)$ then there exist $a_1,a_2\in A$ such that $r$ exceeds the right-hand side of~\eqref{eqn:boissonnat_lemma_1}.
Let $r > \reach(A)$ and let $\tilde{r}\in (\reach(A),r)$.
By the definition of $\reach(A)$ (Definition~\ref{def:reach}), $\exists p \in A_{\tilde{r}}$ and $a_1,a_2\in A$ such that $\norm{a_1-p} = \norm{a_2 - p} = \delta_{A}(p)\leq \tilde{r}$. Since the interior of $B(p, \delta_A(p))$ does not intersect $A$, we have
$$\delta_A\left(\frac{a_1+a_2}2\right) \geq g_{\norm{a_2-a_1}}^{-1}(\delta_A(p)),$$
(see Figure~\ref{fig:proof_theorem_2}) and 
$$g_{\norm{a_2-a_1}}\circ \delta_A\left(\frac{a_1+a_2}2\right) \leq \delta_A(p) \leq \tilde r < r.$$
This proves that $\reach_0(A)=\reach(A)$.
\smallskip

Finally, we show the first equality in~\eqref{eqn:formulation_set_distance}.
The non-decreasing property gives $\lim_{\beta\searrow 0} \reach_\beta (A) \geq \reach_0(A)$. Now, it suffices to show the reverse inequality. For any $\epsilon > 0$, there exists $a_1,a_2\in A$ satisfying
$$g_{\norm{a_2-a_1}}\circ\delta_A\left(\frac{a_1+a_2}2\right) < \reach_0(A) +\epsilon.$$
Any such pair $(a_1,a_2)$ must satisfy $\delta_A\left(\frac{a_1+a_2}2\right) > 0$, and so for $\beta \in \left(0, \delta_A\left(\frac{a_1+a_2}2\right)\right)$, it holds that $\reach_\beta(A) < \reach_0(A) + \epsilon$. Thus, $\lim_{\beta\searrow 0}\reach_\beta(A) \leq \reach_0(A)$.
\end{proof}

\begin{rema}\label{rem:beta_reach_bounded_by_beta}
For any $\alpha\in[0,\infty)$ and $x\in [0,\alpha/2]$, one has $g_{\alpha}(x) \geq x$ (see Figure~\ref{fig:g_alpha_curves}). Thus, for any closed set $A$ and $\beta\in \R^+$,
\begin{equation*}
\reach_{\beta}(A) \geq \beta.
\end{equation*}
Intuitively, this means that the $\beta$-reach excludes small, reach-limiting features of the set that have scale less than $\beta$.

The values of $\beta$ for which $\reach_\beta(A)=\beta$ are the distances from the critical points of the generalized gradient function \citep{chazal2009} to the set $A$.
\end{rema}

\begin{rema}\label{rem:sdr}
Other generalizations of the reach are constructed similarly to the $\beta$-reach, in that they formulate the reach as an infimum or supremum over some set, and add or remove elements in the set using some real parameter (in our case, $\beta$). For example, \cite[Theorem~1]{boissonnat2019} expresses the reach of a set $A\subset\R^d$ as the supremum over a subset of $\R^+$ that satisfies a certain condition relating to $A$. By ``continuously'' weakening the condition with a parameter $\delta\in[0,\infty)$, the authors in \cite{aamari2022} introduce the \textit{spherical distortion radius} as the supremum of the larger subset of $\R^+$ satisfying the weaker condition parametrized by $\delta$. The spherical distortion radius is identified with the reach for $\delta=0$.
In an earlier work, \cite{chazal2009} introduces the \textit{$\mu$-reach} by considering the shortest distance from an element in the set to the \textit{$\mu$-medial axis}, a filtered version of the medial axis by considering regions where the generalized gradient function \citep{lieutier2004} is less than some $\mu\in (0,1]$.
The $\beta$-reach, the $\mu$-reach, and the spherical distortion radius, are all generalizations of the reach that exclude ``small-scale'' features as decided by the corresponding parameter $\beta$, $\mu$, or $\delta$.
\end{rema}

One significant advantage of the $\beta$-reach is its computability for high-dimensional point cloud data (see Section~\ref{sec:point_cloud_beta_reach}). In this setting, the formulation of the reach in terms of the $\beta$-reach for $\beta=0$ can also be used to construct an upper bound for the reach of any compact $A\subset\R^d$ (see Section~\ref{sec:point_cloud_reach}).
\smallskip

%\rya{
%[\cite{chazal2005_1} introduces the \textit{weak feature size} as a generalization of the \textit{local feature size} from \cite{amenta1999}.]
%}

For a closed set $A\subseteq \R^d$, one can study the map $\beta\mapsto \reach_\beta(A)$ for $\beta\in [0,\infty)$, which we refer to as the \textit{$\beta$-reach profile} of the set $A$. We will see in Section~\ref{sec:point_cloud} that a set's $\beta$-reach profile provides pertinent information regarding the estimation of the reach from point cloud data, especially through its first-order approximation at $\beta=0$.

\begin{figure}[t]
    \centering
    \begin{subfigure}{0.45\textwidth}
        \centering
        \includegraphics[width=\textwidth]{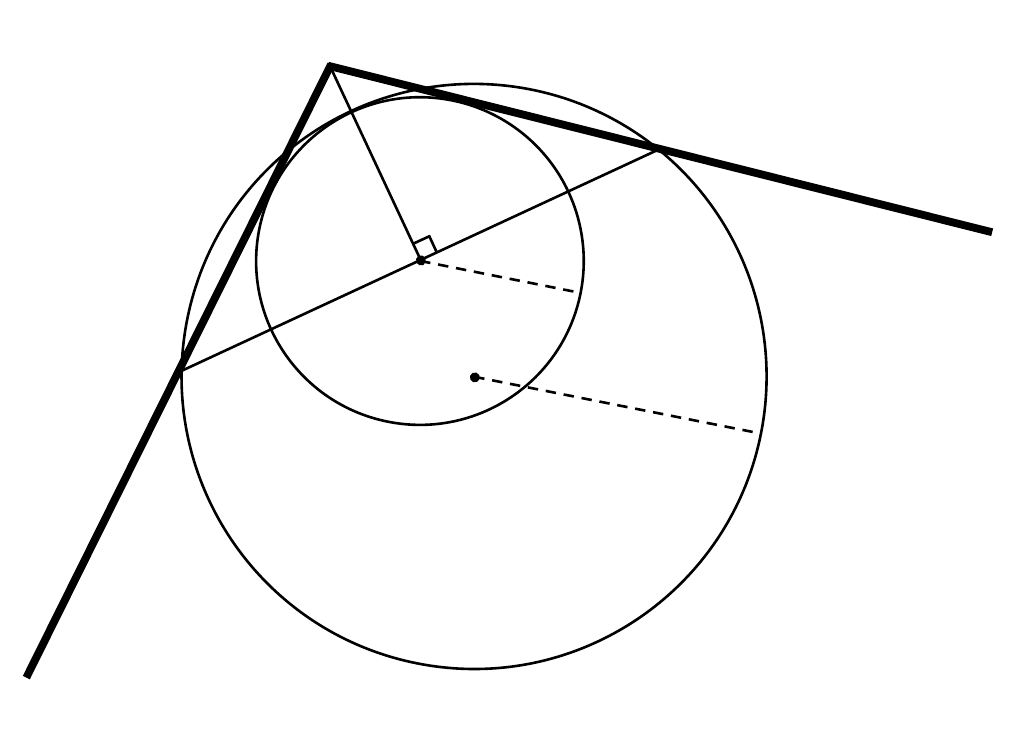}
        \put(-82,62){\small{$\beta$}}
		\put(-150,30){\small{$A$}}
		\put(-85,40){\small{$\reach_\beta(A)$}}
\caption{}
\label{fig:beta_reach_corner}
    \end{subfigure}
    \hfill
    \begin{subfigure}{0.45\textwidth}
        \centering
        \includegraphics[width=\textwidth]{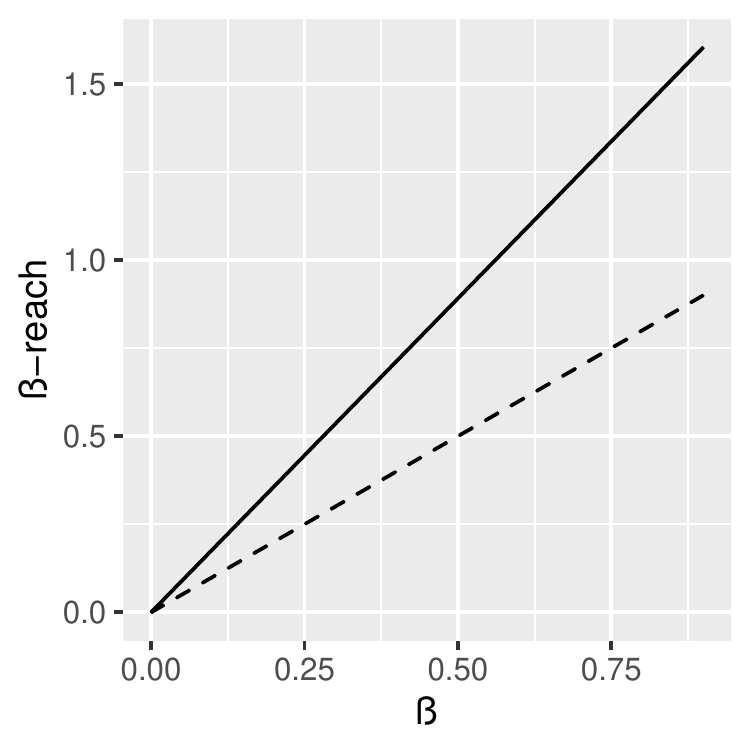}
        \caption{}
        \label{fig:two_segments_plot}
    \end{subfigure}
    \caption{\textbf{(a):} The construction of $\reach_\beta(A)$ for $\beta > 0$ with $A$, the union of two rays in $\R^d$ that originate from the same point. \textbf{(b):} The $\beta$-reach profile of $A$ shown as a solid line. The lower bound with slope 1 is shown as a dashed line.}
    \label{fig:two_segments}
\end{figure}

\begin{figure}[t]
    \centering
    \begin{subfigure}{0.47\textwidth}
        \centering
        \includegraphics[width=\textwidth]{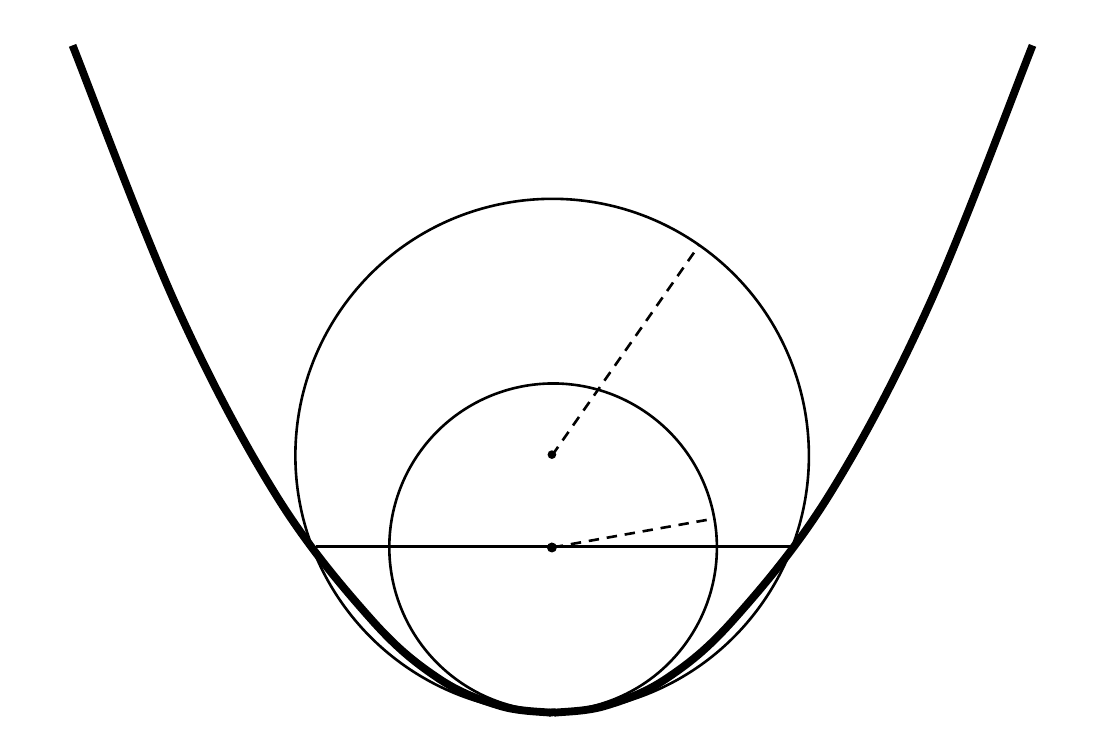}
        \put(-70,35){\small{$\beta$}}
		\put(-153,80){\small{$A$}}
		\put(-106,63){\small{$\reach_\beta(A)$}}
\caption{}
\label{fig:beta_reach_parabola}
    \end{subfigure}
    \hfill
    \begin{subfigure}{0.45\textwidth}
        \centering
        \includegraphics[width=\textwidth]{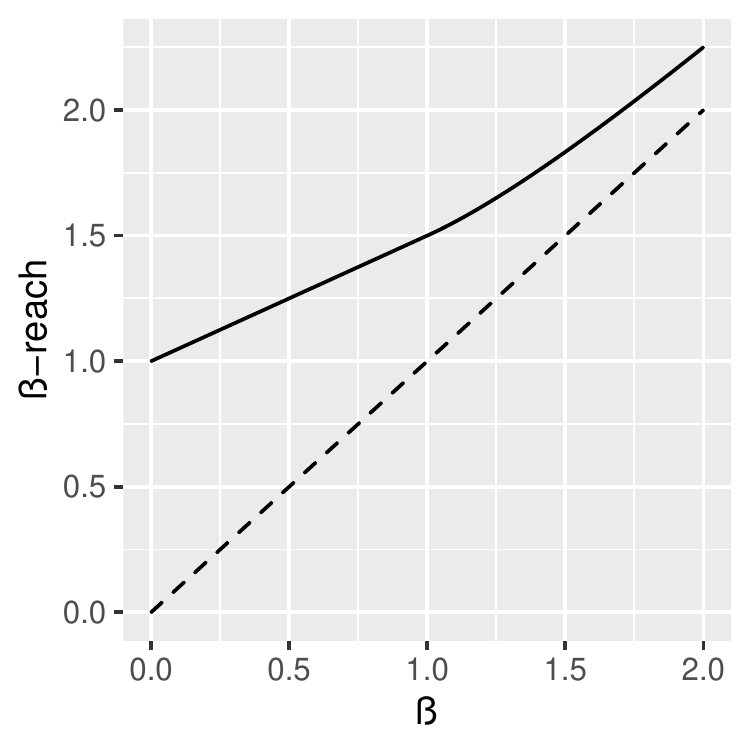}
        \caption{}
        \label{fig:parabola_plot}
    \end{subfigure}
    \caption{\textbf{(a):} The construction of $\reach_\beta(A)$ for $\beta > 0$ with $A$, a paraboloid embedded in $\R^d$. \textbf{(b):} The $\beta$-reach profile of $A$ shown as a solid line. The lower bound with slope 1 is shown as a dashed line.}
    \label{fig:parabola}
\end{figure}

Some exemplary sets and their $\beta$-reach profiles are considered in the several examples that follow.

\begin{exam}[The $\beta$-reach of a corner]\label{exa:line_segments}
For two line segments in $\R^d$ with ends joined by an angle $\theta \in (0,\pi)$, the $\beta$-reach of their union is $\frac\beta 2 \left(1 + \sec^2\left(\frac \theta 2\right)\right)$, for sufficiently small $\beta > 0$. See Figure~\ref{fig:two_segments} for an illustration.
\end{exam}

\begin{exam}[The $\beta$-reach of an arc]
The $\beta$-reach of a circular arc of angle at most $\pi$ is equal to the radius of the arc, for sufficiently small $\beta > 0$.
\end{exam}

\begin{exam}[The $\beta$-reach of a bottleneck structure]
If the reach of a set $A\subset\R^d$ is determined by a bottleneck structure like those described in \cite{aamari2019}, then $\reach_\beta(A) = \reach(A)$ for $\beta\in [0,\reach(A)]$. This is the case, for example, when $A$ is a finite set of points in $\R^d$. 
\end{exam}

\begin{exam}[The $\beta$-reach of a $C^2$-smooth curve]\label{exa:beta_reach_C2}
Let $h:\R\to\R$ be a $C^2$-smooth function with $h'(0) > 0$. Suppose that the graph of the function $f:[-1,1]\to\R$ defined by $f(x) := h(x^2)$ obtains its maximal curvature at $x=0$. Then, the graph $A:= \{(x,f(x)):-1 \leq x\leq 1\}\subset \R^2$ satisfies
\begin{equation}\label{eqn:example_reach}
\reach(A) = \frac{1}{2h'(0)}
\end{equation}
and
\begin{equation}\label{eqn:example_beta_reach}
\reach_\beta(A) = \frac{1}{2h'(0)} + \left(\frac 12 - \frac{h''(0)}{4h'(0)^3}\right)\beta + \littleo(\beta),
\end{equation}
for sufficiently small $\beta>0$.
Justifications for Equations~\eqref{eqn:example_reach} and~\eqref{eqn:example_beta_reach} are provided at the end of Section~\ref{sec:proofs}. Figure~\ref{fig:parabola} depicts a special case of this example with $h(x) = x/2$.
\end{exam}

\section{Methods for point cloud data}\label{sec:point_cloud}

In practice, one might be interested in identifying bounds on the reach and $r$-convexity of a set $A\subseteq\R^d$ given a discrete set of points that are known to reside in $A$.
The goal of this section is to provide computational methods for bounding the reach and the $r$-convexity from above, and producing diagnostics for possible approximations. The three main settings that we consider are:
\begin{enumerate}
\item[(a)] One has access to a point cloud that extends over $\R^d$ in the sense that $\R^d$ can be covered by balls of fixed radius $\epsilon$ centered at each point. Moreover, one knows the partition of the points that lie in $A\subseteq \R^d$, and those that lie in $A^c$. See section~\ref{sec:point_cloud_rconv} for a treatment of this setting.
\item[(b)] One has access to a set of points for which it is known that $A\subseteq \R^d$ can be covered by balls of fixed radius $\epsilon$ centered at each point. Here, $\epsilon$ is known. See section~\ref{sec:point_cloud_reach}.
\item[(c)] The set $A$ is a submanifold of $\R^d$ of dimension $m<d$. One has access to a set of points that is known to be contained in $A$. See section~\ref{sec:point_cloud_beta_reach}.
\end{enumerate}

The mathematical results that we present in this section hold in arbitrary dimension $d$. Nevertheless, the computational complexity of the method described in Section~\ref{sec:point_cloud_rconv} for setting~(a) increases quite drastically as higher dimensions are considered.

For the method that we present for setting~(b) described in Section~\ref{sec:point_cloud_reach}, its computation time depends on the ambient dimension only through the computation of distances in $\R^d$, which is linear in $d$. The algorithm is largely dependent on the number of points used, which may be large in high dimension to ensure small $\epsilon$.

The computational complexity of visualizing the $\beta$-reach profile of the manifold in setting (c) is also linear in the dimension of the ambient space, and so the algorithm that we describe in Section~\ref{sec:point_cloud_beta_reach} can adapt to large values of $d$. Contrast this with existing methods that aim to approximate the reach by first approximating the medial axis \cite{cuevas2014, chazal2005, chazal2009, dey2003, dey2006, lieutier2023}, where in high dimension, accessing the medial axis becomes computationally challenging. The relationship between $d$, $m$, $n$, and our method's performance is elaborated on in Remark~\ref{rem:relationship_n_d_m} in Section~\ref{sec:point_cloud_beta_reach}.

%In this section, we provide methods of bounding the $r$-convexity and the reach of $A$ given such a set of points. We show that as the sampling points become dense in $A$, these upper bounds converge to their respective quantities.
%The Hausdorff metric serves as an excellent candidate to quantify the density of points in the set, since it captures the minimal local density of points within the set. Convergence is thus expressed in terms of the Hausdorff distance from the point cloud to the set of interest tending to 0, as has been done in, for example, \cite{cuevas1997}.

\subsection{An upper bound for the $r$-convexity of a set}\label{sec:point_cloud_rconv}

In this section, we introduce a method for identifying when sets are not $r$-convex, for $r>0$, with a true negative rate of 100\%. That is, given a set $A\subset \R^d$ that is observed over some grid of points (possibly lacking structure), we show how to correctly identify for which values of $r>0$ the set $A$ is \textit{certainly not} $r$-convex. The smallest of these values of $r$ provides an upper bound for $\rconv(A)$, and by Equation~\eqref{eqn:reach_inclusion}, an upper bound for $\reach(A)$. Moreover, if the sampling becomes dense in $\R^d$, we show that this upper bound converges to $\rconv(A)$.
\smallskip

Here, we define operations analogous to dilation and erosion, for discrete sets of points.
\begin{defi}\label{def:pc}
Let $\pc$ be a point cloud in $\R^d$, \textit{i.e.}, a countable subset of $\R^d$. For a set $\hat{A} \subseteq \pc$, we make a slight abuse of notation and denote for $r\in\R$,
$$\hat{A}_r := \begin{cases}
\{p\in \pc : \delta_{\hat{A}}(p) \leq r \}, &\mbox{for } r\geq 0,\\
\{a\in \hat{A} : \delta_{\pc\setminus\hat{A}}(a) > -r\}, &\mbox{for } r < 0.
\end{cases}$$
\end{defi}

\begin{rema}
We emphasize that for a point cloud $\pc$, a subset $\hat{A} \subseteq \pc$, and a real number $r\in\R$, Definition~\ref{def:pc} implies that $\hat{A}_r \subseteq \pc$. Contrast this with the set $\hat{A} \oplus B(0,\vert r\vert)$, which is not contained in $\pc$ for $r > 0$.
\end{rema}

\begin{figure}[t]
\centering
\includegraphics[width=\linewidth]{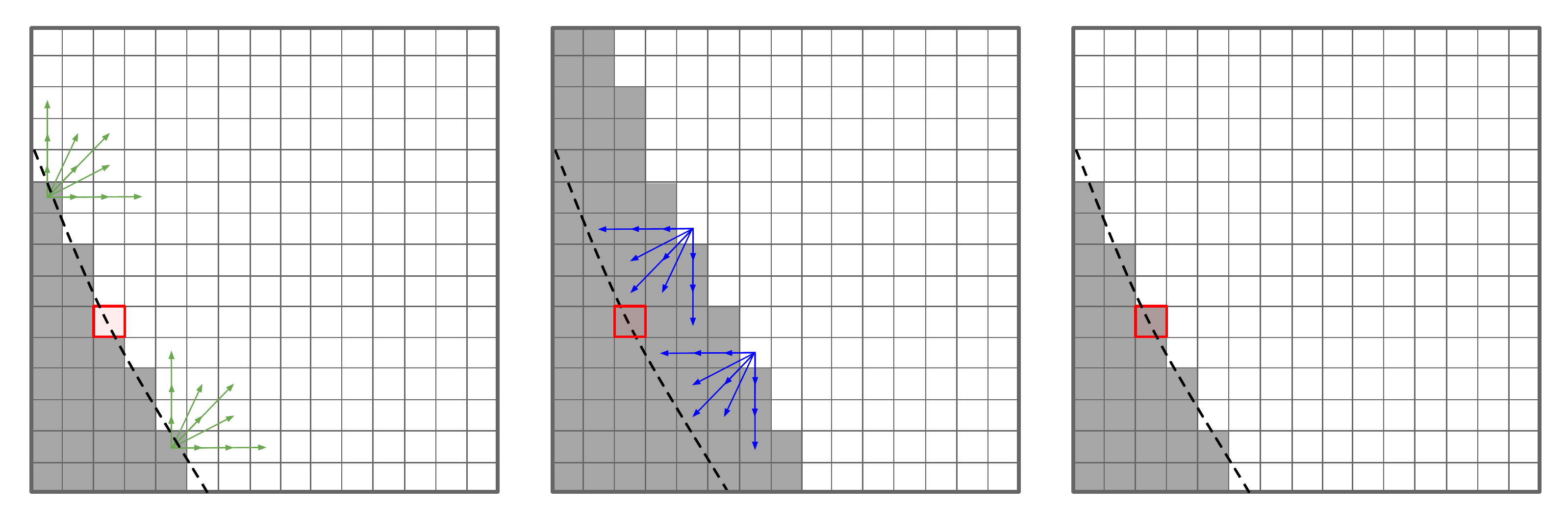}
\put(-60,-10){(c)}
\put(-177,-10){(b)}
\put(-286,-10){(a)}
\caption{An example where $(\hat{A}_r)_{-r}$ in (c) is strictly larger than $\hat{A}$ in (a), while the underlying set $A$ is $r$-convex. The boundary $\partial A$ is shown as a dashed line and $r$ is three times the pixel width.}
\label{fig:discrete_closing}
\end{figure}

Let us illustrate the operations in Definition~\ref{def:pc} using the example in Figure~\ref{fig:discrete_closing}. The set $A\subset\R^2$ occupies the bottom left corner of the domain, just until the dashed line. When sampled on a square lattice $\pc$, one obtains the image in panel (a); each pixel is centered on a point in $\pc$, and grey if the point lies in $A$. Dilating the grey pixels $\hat{A}$ by $r=3\times (pixel\ width)$, one obtains the grey pixels in panel (b), $\hat{A}_r$. If one then erodes $\hat{A}_r$ by $r$ (equivalent to dilating $\pc\setminus \hat{A}_r$ by $r$ and taking complements), one obtains the grey pixels in panel (c), $(\hat{A}_r)_{-r}$. Notice, however, that there is a point in $(\hat{A}_r)_{-r}$ that is not in $\hat{A}$. This might be surprising since we chose a set $A$ such that $A_{\bullet r} = A$. This illustration shows that naively testing for $r$-convexity using the discrete dilation and erosion operations \textit{mistakenly classifies sets as not $r$-convex} when indeed they are.

The following theorem shows how to correctly identify sets as not being $r$-convex, and gives an upper bound for the $r$-convexity of a set that is tight in some sense.

\begin{theo}\label{thm:rconv_estimator}
For $n\in\N^+$, let $\pc^{(n)}$ be a point cloud in $\R^d$, and suppose that
\begin{equation}\label{eqn:epsilon_n_rconv_estimator}
\epsilon_n := \sup\left\{\delta_{\pc^{(n)}}(q) : q\in\R^d\right\}
\end{equation}
is finite.
Let $A$ be a compact subset of $\R^d$, and for $n\in\N^+$, define $\hat{A}^{(n)} := A \cap \pc^{(n)}$ and the corresponding bound,
\begin{equation}\label{eqn:convexity_estimator}
\rnv{A}{n} := \inf\left\{r > \epsilon_n : (\hat{A}^{(n)}_{r-\epsilon_n})_{-(r+\epsilon_n)} \nsubseteq \hat{A}^{(n)}\right\}.
\end{equation}
It holds that 
\begin{equation}\label{eqn:rconv_bound}
\inf_{n\in\N^+}\rnv{A}{n} \geq \rconv(A).
\end{equation}
Moreover, if $\epsilon_n\to 0$ and $d_H(\hat{A}^{(n)},A)\to 0$ as $n\to\infty$, then
\begin{equation}\label{eqn:rconv_lim}
\lim_{n\to \infty}\rnv{A}{n} = \rconv(A).
\end{equation}
\end{theo}

\begin{figure}[t]
\centering
\includegraphics[width=0.8\linewidth]{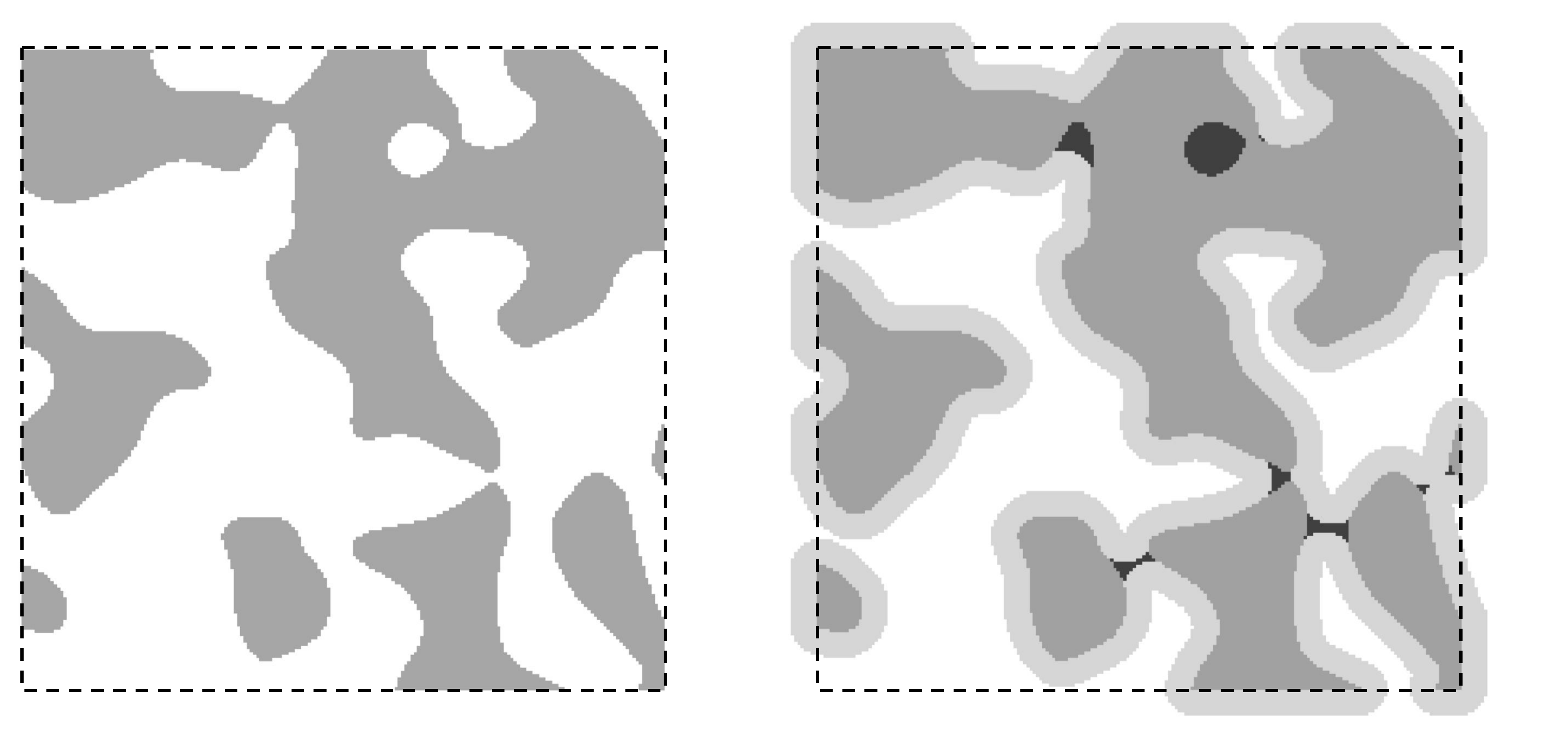}
\put(-80,-10){(b)}
\put(-220,-10){(a)}
\caption{\textbf{(a):} A set $A$ is shown as a pixelated image in grey. The set of grey pixel centers is $\hat A^{(n)}$. \textbf{(b):} The set $(\hat{A}^{(n)}_{r-\epsilon_n})_{-(r+\epsilon_n)}$ contains the dark grey pixels, which are not contained in $\hat A^{(n)}$. Thus, $r \geq \rconv(A)$. Here, $r$ is nine times the pixel width.}
\label{fig:corrected_reach_test}
\end{figure}

Theorem~\ref{thm:rconv_estimator} provides a method for correctly identifying which subsets of $\R^d$ are not $r$-convex. If a dilation of the the discretization $\hat{A}^{(n)}$ by $r-\epsilon_n$ followed by an erosion of $r+\epsilon_n$ produces a set that is not contained in $\hat{A}^{(n)}$, then $\rconv(A) \leq r$.
See Figure~\ref{fig:corrected_reach_test} for an example of a set $A$ for which $(\hat{A}^{(n)}_{r-\epsilon_n})_{-(r+\epsilon_n)} \nsubseteq \hat{A}^{(n)}$. The precise regions where $r$-convexity does not hold are highlighted by the method.

\begin{rema}
Note that $\rnv{A}{n}$ in~\eqref{eqn:convexity_estimator} can be computed using entirely available information, since $\epsilon_n$ in~\eqref{eqn:epsilon_n_rconv_estimator} is a feature of the point cloud $\pc^{(n)}$ and not of the unknown set $A$. In particular, there is no need to estimate $d_{H}(\hat{A}^{(n)},A)$ for the construction of $\rnv{A}{n}$. A binary search algorithm, along with the discrete dilation operations in Definition~\ref{def:pc}, are sufficient for the numerical calculation of $\rnv{A}{n}$.
\end{rema}

\begin{rema}
The requirement that $\pc^{(n)}$ extends over all of $\R^d$ (see Equation~\eqref{eqn:epsilon_n_rconv_estimator}) allows for a mathematical simplification and is not needed in practice. In real applications, one only requires that $d_H(\pc^{(n)},T) = \epsilon_n$ for some compact $T\subset\R^d$ that contains $A_{\rconv(A)+2\epsilon_n}$. In the case where $A$ is lower dimensional than the embedding space $\R^d$, one can add any finite number of points uniformly distributed on $T$ to $\pc^{(n)}$, and they will almost surely not belong to $A$.
\end{rema}

\begin{figure}[t]
\centering
\includegraphics[width=0.5\linewidth]{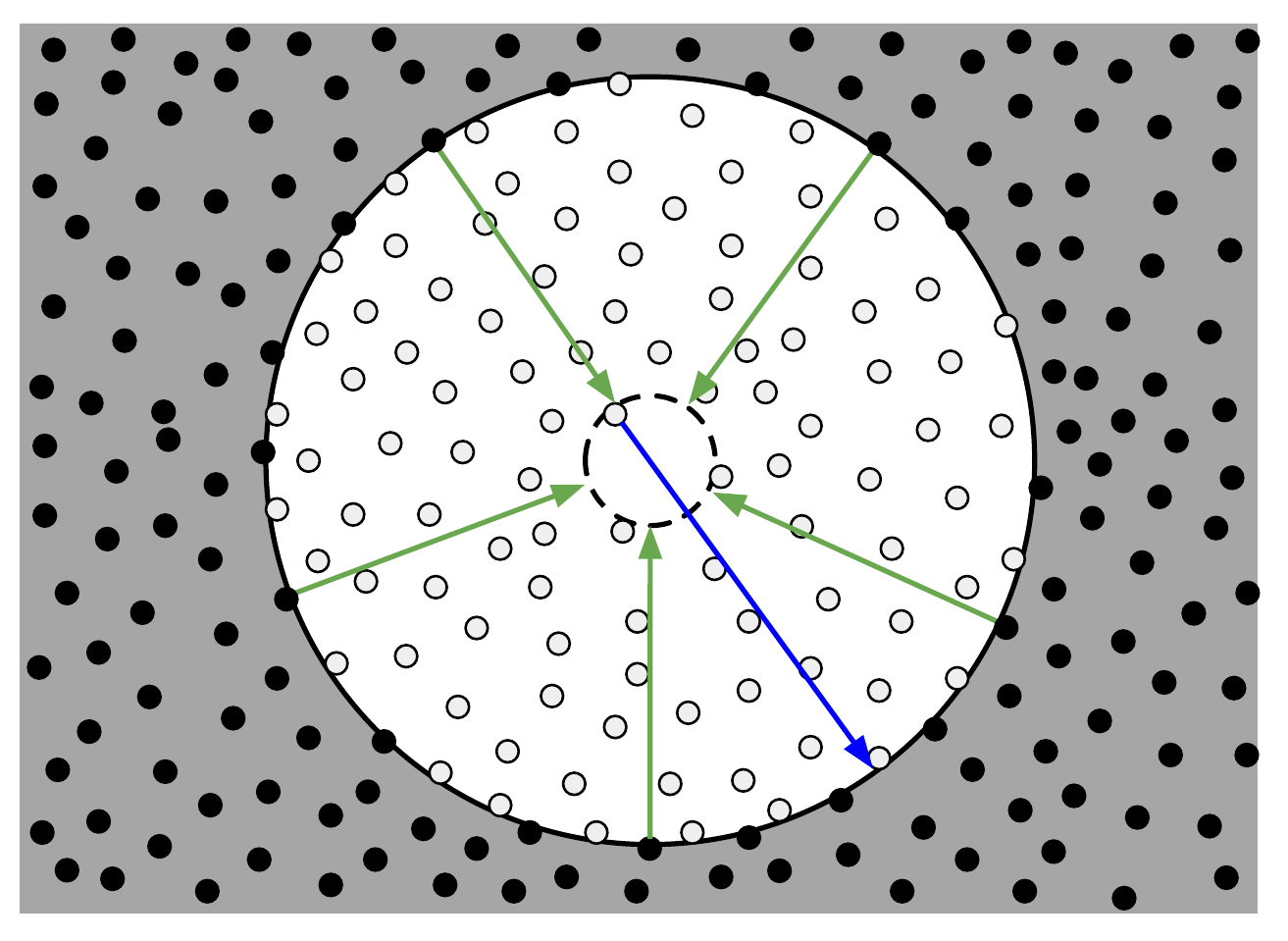}
\caption{A discrete dilation of the set $\hat A^{(n)}$ (black points) by the length of the green arrows ($\rconv(A) - \epsilon_n$) followed by a discrete erosion by the length of the longer blue arrow ($\rconv(A) + \epsilon_n$) leaves all of the white points out of the resulting set. The boundary of $A_{\rconv(A)-\epsilon_n}$ is shown as a dashed line.}
\label{fig:bad_case_reach}
\end{figure}

The proof of Theorem~\ref{thm:rconv_estimator} is quite technical, and so we postpone it to Section~\ref{sec:proofs}. Nevertheless, Figure~\ref{fig:bad_case_reach} provides an intuitive illustration that helps understand Equation~\eqref{eqn:rconv_bound}. This is an example of a worst-case scenario, in that a dilation of $\hat A^{(n)}$ by any more than $\rconv(A)-\epsilon_n$ results in all of $\pc^{(n)}$ being consumed. The dilation radius $\rconv(A)-\epsilon_n$ is maximal in the sense that, by dilating any less, there is guaranteed to be at least one element of $\pc^{(n)}$ that is not in the dilation. The distance between this one remaining element of $\pc^{(n)}$ and the other points in $\pc^{(n)}\setminus \hat{A}^{(n)}$ might be anywhere up to $r+\epsilon_n$, where $r-\epsilon_n \in(0, \rconv(A)-\epsilon_n)$ is the dilation radius. Therefore, an erosion by at least $r+\epsilon_n$ is necessary to guarantee that the resulting set is a subset of $\hat{A}^{(n)}$. In this sense, the bound in~\eqref{eqn:rconv_bound} is tight.

The following result from \cite{rodriguezcasal2016}, while interesting on its own, is instrumental in the proof of Equation~\eqref{eqn:rconv_lim}.

\begin{prop}[Lemma~8.3 in \cite{rodriguezcasal2016}]\label{prp:open_subset_rconv}
Let $r\in\R^+$ and let $A\subset \R^d$ be a closed set satisfying $\rconv(A) < r$. The set $A_{\bullet r} \setminus A$ contains an open subset of $\R^d$. 
\end{prop}

The relationship between Proposition~\ref{prp:open_subset_rconv} and Equation~\eqref{eqn:rconv_lim} is that, for $r>\rconv(A)$, as the point cloud $\pc^{(n)}$ becomes more dense in $\R^d$, the open subset in $A_{\bullet r} \setminus A$ fills with points that remain in $(\hat{A}^{(n)}_{r-\epsilon_n})_{-(r+\epsilon_n)}$ for $\epsilon_n$ sufficiently small. Thus,
\begin{equation*}
\lim_{n\to \infty}\rnv{A}{n} < r,
\end{equation*}
and since $r\in(\rconv(A),\infty)$ is arbitrary, the limit is at most $\rconv(A)$. Equality then holds by Equation~\eqref{eqn:rconv_bound}. The full proof of Theorem~\ref{thm:rconv_estimator} and an alternative proof of Proposition~\ref{prp:open_subset_rconv} are provided in Section~\ref{sec:proofs}.

\begin{rema}\label{rem:no_rate_rconv}
Equation~\eqref{eqn:rconv_lim} is provided without a rate of convergence. This is due to the fact that Proposition~\ref{prp:open_subset_rconv} provides no guarantees on the size of the open subset that can be found in $A_{\bullet r}\setminus A$ for $r > \rconv(A)$. To provide a rate of convergence, a deeper analysis is needed to understand the rate at which the size of the largest open ball in $A_{\bullet r}\setminus A$ decreases as $r \searrow \rconv(A)$. Conversely, the bound that we construct for the reach in Section~\ref{sec:point_cloud_reach} converges to the reach at a known rate (see Theorem~\ref{thm:reach_estimator} below).
\end{rema}

%Note that if $\pc$ is a square lattice, then $\epsilon = \delta/\sqrt{2}$, where $\delta$ is the lattice spacing.

\subsection{An upper bound for the reach of a set}\label{sec:point_cloud_reach}

We have already seen that $\rnv An$ in~\eqref{eqn:convexity_estimator} is an upper bound for the reach by Equations~\eqref{eqn:reach_inclusion} and~\eqref{eqn:rconv_bound}. In this section, we introduce another computable bound for the reach based on the expression for the $\beta$-reach in Definition~\ref{def:beta_reach}. The context in which one can apply this bound was introduced as setting (b) at the start of Section~\ref{sec:point_cloud}: a countable set of points $\hat{A}^{(n)}$ is known to be included in a compact set $A\subset\R^d$, and its Hausdorff distance to $A$ is known to be at most $\epsilon_n$.

A very weak regularity condition is imposed on the set $A$ in terms of its $\beta$-reach profile near 0; it is used to prove the rate at which our upper bound converges to $\reach(A)$ as $\hat{A}^{(n)}\to A$ in the Hausdorff metric.

\begin{assu}\label{ass:beta_reach_increasing}
Suppose that the set $A\subset\R^d$ is compact, and that there exists $\delta > 0$ such that the map $\beta \mapsto \reach_\beta(A)$ for $\beta\in \R$ is either constant or strictly increasing on $[0,\delta]$. In addition, suppose that
\begin{equation}\label{eqn:K_A}
K_A := \lim_{\beta\searrow 0}\frac{\reach_{\beta}(A) - \reach(A)}{\beta}
\end{equation}
exists and is finite.
\end{assu}

\begin{figure}[t]
    \centering
    \begin{subfigure}{0.45\textwidth}
        \centering
        \includegraphics[width=\textwidth]{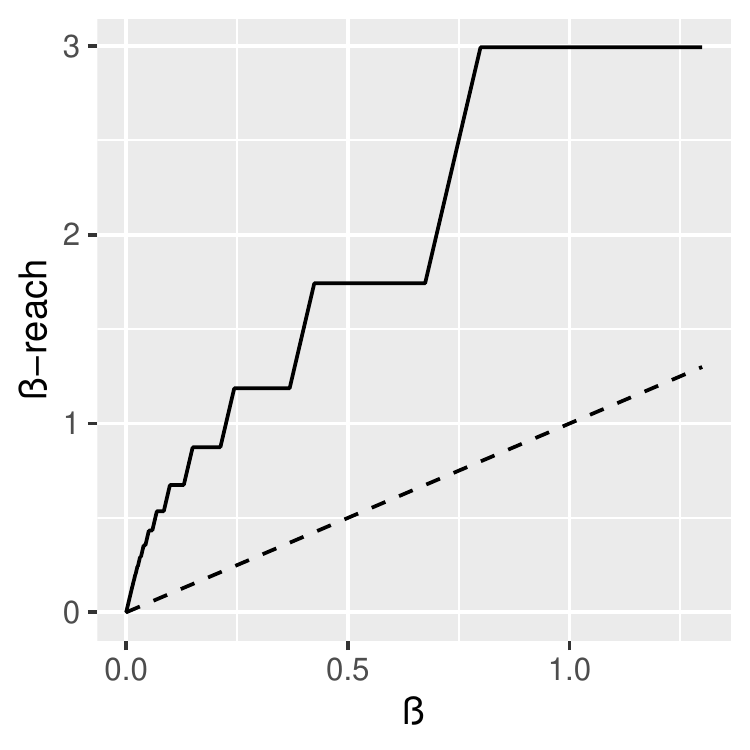}
        \caption{}
        \label{fig:maliscious_steps}
    \end{subfigure}
    \hfill
    \begin{subfigure}{0.45\textwidth}
        \centering
        \includegraphics[width=\textwidth]{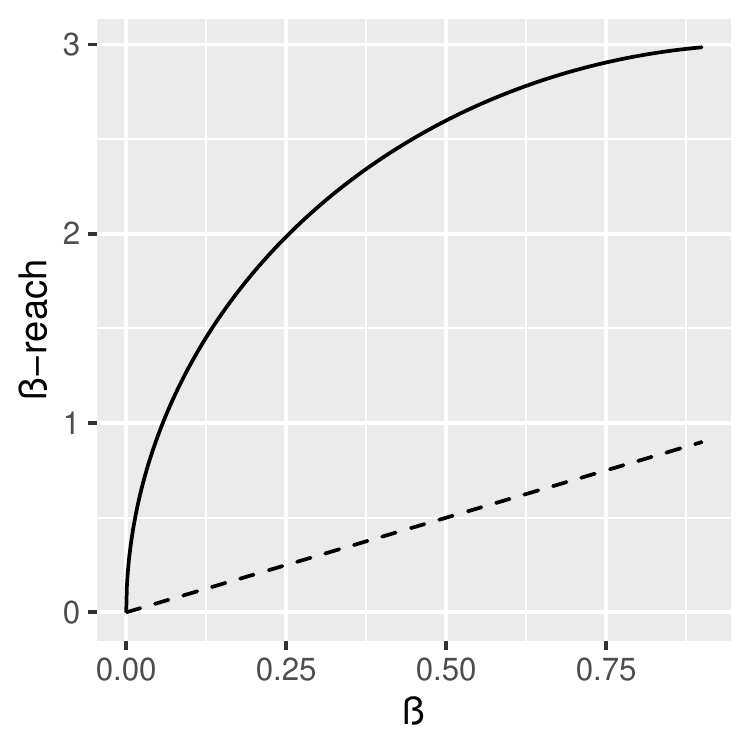}
        \caption{}
        \label{fig:maliscious_round}
    \end{subfigure}
    \caption{Hypothetical $\beta$-reach profiles of sets that would \textit{not} satisfy Assumption~\ref{ass:beta_reach_increasing}. \textbf{(a):} The $\beta$-reach profile is neither constant nor strictly increasing on any neighbourhood of 0. \textbf{(b):} The limit in~\eqref{eqn:K_A} is infinite.}
    \label{fig:maliscious}
\end{figure}

Assumption~\ref{ass:beta_reach_increasing} is imposed on the $\beta$-reach of $A$, as opposed to the regularity of $\partial A$, since the $\beta$-reach lends very naturally to the proof of Theorem~\ref{thm:reach_estimator}. Moreover, Assumption~\ref{ass:beta_reach_increasing} allows for most compact sets in $\R^d$, excluding some pathological counterexamples such as hypothetical sets admitting the $\beta$-reach profiles in Figure~\ref{fig:maliscious}. Assumption~\ref{ass:beta_reach_increasing} is seen to hold in all of Examples~\ref{exa:line_segments} through~\ref{exa:beta_reach_C2} in Section~\ref{sec:definitions_reach_beta_reach}.

\begin{theo}\label{thm:reach_estimator}
Let $A\subseteq\R^d$ be closed. For each $n\in\N^+$, let $\hat{A}^{(n)}$ be a countable subset of $A$. Suppose that $\hat A^{(n)}\to A$ in the Hausdorff metric. \textit{i.e.}, there is a sequence $(\epsilon_n)_{n\geq 1}$ in $\R^+$ tending to 0 as $n\to \infty$ such that $\epsilon_n \geq d_H(\hat{A}^{(n)},A)$ for all $n\in\N^+$.
For each $n\in\N^+$, define the corresponding bound on $\reach(A)$ by
\begin{align}\label{eqn:reach_estimator}
\rch{A}{n} := \inf\bigg\{g_{\norm{a_2-a_1}}(x-\epsilon_n) :\ & a_1,a_2 \in \hat{A}^{(n)},\\
&x=\delta_{\hat{A}^{(n)}}\left(\frac{a_1+a_2}2\right) \geq \epsilon_n\bigg\},\nonumber
\end{align}
where $g_{\norm{a_2-a_1}}$ is defined in~\eqref{eqn:g_ell}.
Then, 
\begin{equation}\label{eqn:rch_bound}
\rch{A}{n} \geq \reach(A),
\end{equation}
for $n\in\N^+$. Furthermore, if $A$ satisfies Assumption~\ref{ass:beta_reach_increasing} and is not convex, then there exists $n_0\in\N^+$ such that
\begin{equation}\label{eqn:rch_converge_rate}
\rch{A}{n} - \reach(A) \leq (2\,\reach(A) + K_A)\sqrt{\epsilon_n},
\end{equation}
for all $n\geq n_0$, where $K_A$ is defined as in~\eqref{eqn:K_A}.
\end{theo}

\begin{proof}[Proof of Theorem~\ref{thm:reach_estimator}]
We begin by showing~\eqref{eqn:rch_bound} in the case where $A$ is not convex (for $A$ convex, the result is trivial since $\rch{A}{n} = \reach(A)=\infty$ \citep{federer1959}). Suppose that for some fixed $n\in \N^+$, we have $\epsilon_n \geq d_H(\hat{A}^{(n)},A)$. We write,
\begin{align}\label{eqn:proof_reach_estimator_inequality_1}
\reach_0(A) &= \inf\bigg\{g_{\norm{a_2-a_1}}\circ \delta_A\Big(\frac{a_1+a_2}2\Big) : a_1,a_2 \in A\bigg\}\nonumber\\
&\leq \inf\bigg\{g_{\norm{a_2-a_1}}\circ \delta_A\Big(\frac{a_1+a_2}2\Big) :
\begin{aligned}[t]
&a_1,a_2 \in \hat{A}^{(n)}\\
&\delta_{\hat{A}^{(n)}}\Big(\frac{a_1+a_2}2\Big) \geq \epsilon_n\bigg\}.
\end{aligned}
\end{align}
The equality in~\eqref{eqn:proof_reach_estimator_inequality_1} is an application of Definition~\ref{def:beta_reach}, and the inequality holds since the infimum is taken over a smaller subset. For any $a_1,a_2\in A$, there exists a projection of $\frac{a_1+a_2}2$ onto $A$, namely $a^\pi\in A$, satisfying
$$\norm{a^\pi - \frac{a_1+a_2}2} = \delta_A\Big(\frac{a_1+a_2}2\Big).$$
By the triangle inequality and the fact that $\epsilon_n \geq d_H(\hat{A}^{(n)},A)$,
$$\delta_{\hat{A}^{(n)}}\Big(\frac{a_1+a_2}2\Big) \leq \delta_{\hat{A}^{(n)}}(a^\pi) + \norm{a^\pi - \frac{a_1+a_2}2} \leq \epsilon_n + \delta_A\Big(\frac{a_1+a_2}2\Big).$$
Rearranging gives,
\begin{equation}\label{eqn:any_two_as}
\delta_{\hat{A}^{(n)}}\Big(\frac{a_1+a_2}2\Big) - \epsilon_n \leq \delta_A\Big(\frac{a_1+a_2}2\Big).
\end{equation}
Since $g_{\norm{a_2-a_1}}$ is non-increasing, the rightmost expression in~\eqref{eqn:proof_reach_estimator_inequality_1} cannot decrease when $\delta_A\big(\frac{a_1+a_2}2\big)$ is replaced by $\delta_{\hat{A}^{(n)}}\big(\frac{a_1+a_2}2\big) - \epsilon_n$. Thus,
$$\rch{A}{n} \geq \reach_0(A) = \reach(A),$$
where the latter equality is an application of Theorem~\ref{thm:equivalent_reach}.
\smallskip

\begin{figure}[t]
\centering
\includegraphics[width=0.7\linewidth]{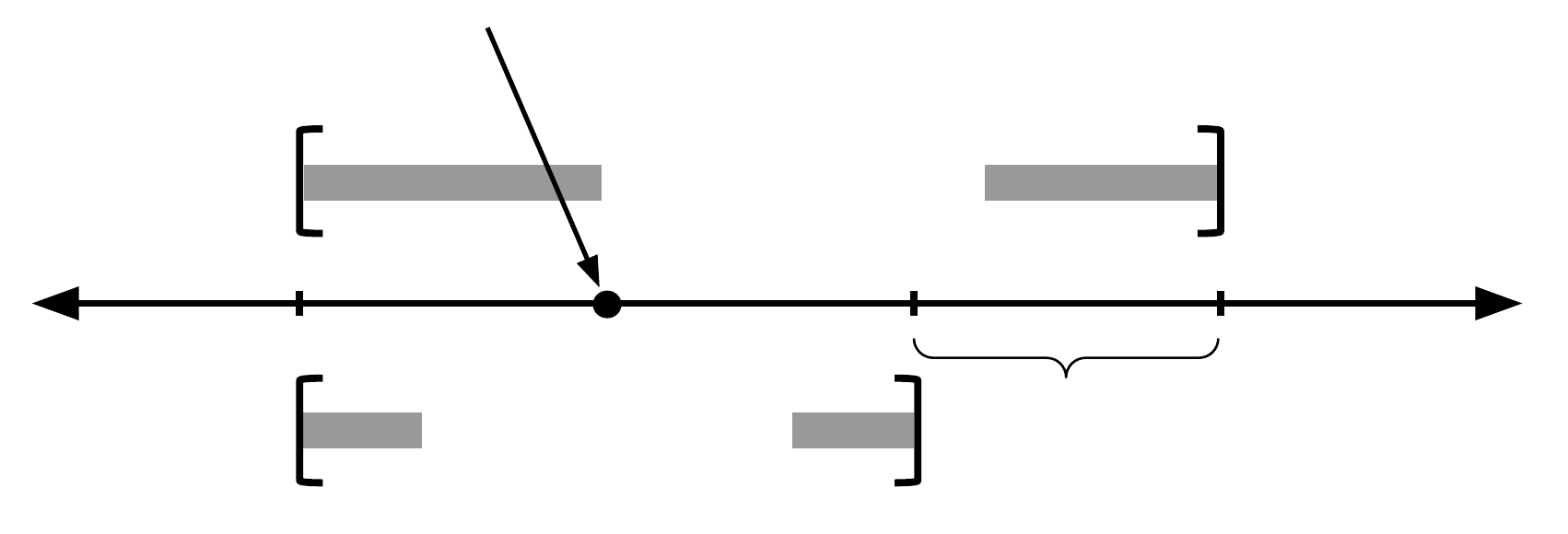}
\put(-195,83){\small{$\delta_A\big(\frac{a_1+a_2}2\big)$}}
\put(-145,50){\small{$\delta_{\hat{A}^{(n)}}\big(\frac{p_1+p_2}2\big)$}}
\put(-167,13){\small{$\delta_A\big(\frac{p_1+p_2}2\big)$}}
\put(-80,15){\small{$\epsilon_n$}}
\caption{Visual aid for Equation~\eqref{eqn:case_1}. The real numbers $\delta_{\hat{A}^{(n)}}\big(\frac{p_1+p_2}2\big)$ and $\delta_A\big(\frac{p_1+p_2}2\big)$ are contained in their respective intervals, positioned relative to $\delta_A\big(\frac{a_1+a_2}2\big)$. Ticks have a spacing of $\epsilon_n$.}
\label{fig:case_1}
\end{figure}
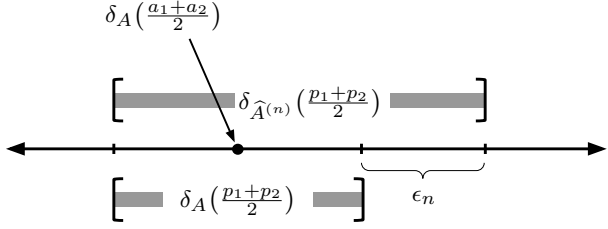

Now, we proceed to show~\eqref{eqn:rch_converge_rate}. Recall the equality in~\eqref{eqn:proof_reach_estimator_inequality_1}, and split the analysis into two cases.
\begin{description}
\item[Case 1:] \emph{There exists $a_1,a_2\in A$ such that $a_1\neq a_2$ and $\reach(A)=g_{\norm{a_2-a_1}}\circ \delta_A\left(\frac{a_1+a_2}2\right)$.}\\
For this pair $a_1,a_2\in A$ there exists $p_1,p_2\in \hat{A}^{(n)}$ satisfying $\norm{p_2-a_2}, \norm{p_1-a_1} \leq \epsilon_n$. This implies
\begin{equation}\label{eqn:midpoint_distances}
\norm{\frac{p_1+p_2}2 - \frac{a_1 + a_2}2} \leq \epsilon_n.
\end{equation}
Equation~\eqref{eqn:midpoint_distances} tells us that the two midpoints are close, and so their distances to $A$ cannot differ by more than the distance between them. \textit{i.e.},
\begin{equation*}
\left\vert\delta_A\left(\frac{a_1+a_2}2\right)
- \delta_A\left(\frac{p_1+p_2}2\right)\right\vert \leq \epsilon_n.
\end{equation*}
Since Equation~\eqref{eqn:any_two_as} holds for $p_1,p_2\in A$, and $\delta_A\left(\frac{p_1+p_2}2\right) \leq \delta_{\hat{A}^{(n)}}\left(\frac{p_1+p_2}2\right)$, we also have
\begin{equation}\label{eqn:case_1}
\delta_A\left(\frac{a_1+a_2}2\right) - \epsilon_n \leq \delta_{\hat{A}^{(n)}}\left(\frac{p_1+p_2}2\right) \leq \delta_A\left(\frac{a_1+a_2}2\right) + 2\epsilon_n,
\end{equation}
by the triangle inequality (see Figure~\ref{fig:case_1}).

Now, suppose that $n$ is sufficiently large such that $2\epsilon_n < y := \delta_A\left(\frac{a_1+a_2}2\right)$.
Then, by~\eqref{eqn:case_1}, $x:=\delta_{\hat{A}^{(n)}}\big(\frac{p_1+p_2}2\big)-\epsilon_n > 0$ and so
\begin{align}\label{eqn:bounding_rch_by_as_and_y}
\rch An &\leq g_{\norm{p_2-p_1}}(x) = \frac{\norm{p_2-p_1}^2}{8x} + \frac{x}{2}\nonumber\\
&\leq \frac{(\norm{a_2-a_1} + 2\epsilon_n)^2}{8x} + \frac{x}{2}\nonumber\\
&\leq \frac{(\norm{a_2-a_1} + 2\epsilon_n)^2}{8(y-2\epsilon_n)} + \frac{y+\epsilon_n}{2}.
\end{align}
The second inequality in~\eqref{eqn:bounding_rch_by_as_and_y} holds since $\norm{p_2-p_1} \leq \norm{a_2-a_1} + 2\epsilon_n$, and the final inequality holds by Equation~\eqref{eqn:case_1}.

Recall that $\reach(A)=\frac{\norm{a_2-a_1}^2}{8y} + \frac{y}{2}$. The difference between this and the final expression in~\eqref{eqn:bounding_rch_by_as_and_y} is of the order $\mathcal{O}(\epsilon_n)$, which is stronger than required.
\smallskip

\item[Case 2:] \emph{$\reach(A) < g_{\norm{a_2-a_1}}\circ \delta_A\left(\frac{a_1+a_2}2\right)$ for all distinct $a_1,a_2\in A$.}\\
By Assumption~\ref{ass:beta_reach_increasing}, for all sufficiently large $n$, the function $\reach_\beta(A)$ is strictly increasing for $\beta$ in a neighbourhood of $\beta_n := \sqrt{\epsilon_n}+2\epsilon_n$. For each of these values of $n$, define
$$\Psi_n := \{(a_1,a_2)\in A^2 : \delta_A\left(\frac{a_1+a_2}2\right) \geq \beta_n\},$$
which is compact by Assumption~\ref{ass:beta_reach_increasing},
and the continuous map $f_n:\Psi_n\to\R^+$ by
$$f_n\big((\psi_1,\psi_2)\big) = g_{\norm{\psi_2-\psi_1}}\circ \delta_A\left(\frac{\psi_1+\psi_2}2\right).$$
Remark that $\reach_{\beta_n}(A) = \inf_{\psi\in\Psi_n}f_n(\psi)$, and so there is an element $\psi^*_n := (\psi^*_{1n}, \psi^*_{2n})\in\Psi_n$ for which
$$\reach_{\beta_n}(A) = f_n(\psi^*_n).$$
By Definition~\ref{def:beta_reach}, $\reach_\beta(A)\leq f_n(\psi^*_n)$ for $\beta = \delta_A\big(\frac{\psi^*_{1n}+\psi^*_{2n}}{2}\big)$. The $\beta$-reach profile of $A$ is thus constant on the interval $\big[\beta_n,\delta_A\big(\frac{\psi^*_{1n}+\psi^*_{2n}}{2}\big)\big]$ by Theorem~\ref{thm:equivalent_reach}, but by Assumption~\ref{ass:beta_reach_increasing}, any such interval must have length 0. Therefore,
\begin{equation*}
\beta_n = \delta_A\left(\frac{\psi^*_{1n}+\psi^*_{2n}}{2}\right).
\end{equation*}

Again by Assumption~\ref{ass:beta_reach_increasing},
\begin{align*}
\reach_{\beta_n}(A) = \frac{\norm{\psi^*_{2n} - \psi^*_{1n}}^2}{8\beta_n} + \frac{\beta_n}2 = \reach(A) + K_A \beta_n + \littleo(\beta_n).
\end{align*}
Therefore,
\begin{align*}
\norm{\psi^*_{2n} - \psi^*_{1n}}^2 &= 8\beta_n\,\reach(A) + (8K_A-4)\beta_n^2 + \littleo(\beta_n^2)\\
&= 8\sqrt{\epsilon_n}\,\reach(A) + (16\,\reach(A) + 8K_A -4)\epsilon_n + \littleo(\epsilon_n).
\end{align*}
As seen in Case~1, Equation~\eqref{eqn:bounding_rch_by_as_and_y},
\begin{align*}
\rch An &\leq \frac{(\norm{\psi^*_{2n} - \psi^*_{1n}} + 2\epsilon_n)^2}{8(\beta_n - 2\epsilon_n)} + \frac{\beta_n+\epsilon_n}2\\
&= \reach(A) + \left(2\,\reach(A) + K_A\right)\sqrt{\epsilon_n} + \littleo(\sqrt{\epsilon_n}).
\end{align*}
\end{description} 
\end{proof}

\begin{rema}\label{rem:aamari_split}
The split of the proof of Theorem~\ref{thm:reach_estimator} into two cases is natural in the context of reach estimation. Recall Theorem~3.4 in \cite{aamari2019}: For a compact set $A$ with $0 < \reach(A) < \infty$, there is either a bottleneck structure (two distinct points $a_1,a_2\in A$ such that $\norm{a_2-a_1} = 2\delta_A\left(\frac{a_1+a_2}2\right) = 2\,\reach(A)$) or an arc-length parametrized geodesic of $A$ with curvature $1/\reach(A)$. The convergence rate developed in Case~1 applies to sets whose reach is decided by a bottleneck structure, and that of Case~2 applies to sets whose reach is decided by a region of high curvature.
\end{rema}

\begin{rema}
In \cite{aamari2022}, the authors provide a minimax optimal rates of convergence for statistical estimators of the reach. The authors proceed by assuming that $A\subset\R^d$ is a $C^k$-smooth manifold without boundary, with $k\geq 3$. The minimax rates in \cite[Theorem~6.6]{aamari2022} are expressed in terms of $k$, the number of continuous derivatives of the underlying manifold.

It is not surprising that our bound on the reach---for sets satisfying the very general Assumption~\ref{ass:beta_reach_increasing}---does not converge to the reach at the optimal rates given in \cite{aamari2022}. Nevertheless, we do observe a similar phenomenon in that the rate of convergence derived for a set whose reach is determined by a bottleneck structure ($\bigo(\epsilon_n)$ in Case~1 in the proof of Theorem~\ref{thm:reach_estimator}) is faster than the rate when the reach is determined by a region of high curvature ($\bigo(\sqrt{\epsilon_n})$ in Case~2).
\end{rema}

\subsection{The $\beta$-reach profiles of high-dimensional point clouds}\label{sec:point_cloud_beta_reach}

Here, we tackle the problem of approximating the reach of smooth manifolds of low dimension embedded in high-dimensional Euclidean space. As discussed in the introduction, this setting is the focus of many works. Given a $C^1$-smooth manifold $M\subset \R^d$ of dimension $m<d$, we suppose that one has access to a set of points $\hat{M}\subset M$ from which one would like to infer $\reach(M)$.

Recall from Definition~\ref{def:beta_reach} that the $\beta$-reach of a point cloud is easily computable, and the computation time does not heavily depend on the dimension $d$ of the ambient space. The $\beta$-reach profile of a point cloud can be obtained by considering pairs of points in the point cloud and computing the distance from their midpoint to the closest neighbour in the point cloud. For a fixed number of points, the computation time for calculating the \textit{exact} $\beta$-reach profile of the point-cloud scales linearly in $d$. Of course, the $beta$-reach profile of $\hat{M}$ is not identical to that of $M$, however, we argue here that it provides a good approximation for the values of $\beta$ away from 0.

Recall that the bounds developed in Sections~\ref{sec:point_cloud_rconv} and~\ref{sec:point_cloud_reach} are for subsets of $\R^d$ that possibly have positive $d$-volume. This means that a point cloud $\hat A$ contained in a set $A\subset\R^d$ can recede a distance $\epsilon:=d_H(\hat A,A)$ from the topological boundary of $A$. Thus, the distance from a point in $A^c$ to $A$ can be up to $\epsilon$ \textit{shorter} than the distance to the nearest point in $\hat A$. Conversely, for a smooth manifold $M$ embedded in $\R^d$, projection of an element $p\in\R^d$ onto $M$ is along a vector normal to $M$ at the projection in $M$. Thus, the distance from $p$ to its nearest point in $\hat{M}$ is more similar to the length of the projection vector, so in this setting, the bounds constructed previously are overly robust. Moreover, existing algorithms for computing a mesh from point cloud data can be leveraged to improve quality of the estimate of the distance to the nearest point in $M$, thus improving the estimate of the $\beta$-reach.
\smallskip

\begin{figure}[t]
\centering
\includegraphics[width=0.95\linewidth]{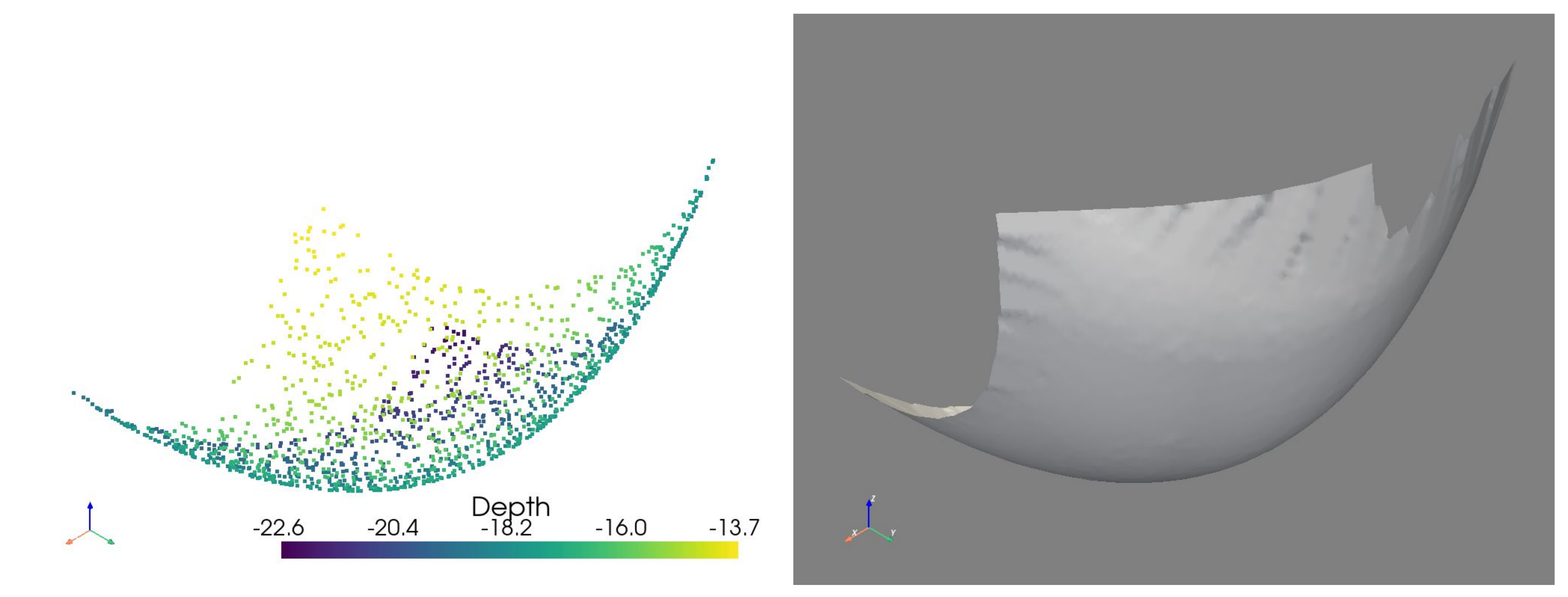}
\put(-90,-10){(b)}
\put(-250,-10){(a)}
\caption{\textbf{(a):} The point cloud $\hat M$ in Example~\ref{exa:point_cloud_surface}. \textbf{(b):} The mesh $\mathcal{M}(\hat M)$ computed with the Python library \texttt{pyvista}.}
\label{fig:point_cloud_surface}
\end{figure}

\begin{figure}[t]
\centering
\includegraphics[width=0.9\linewidth]{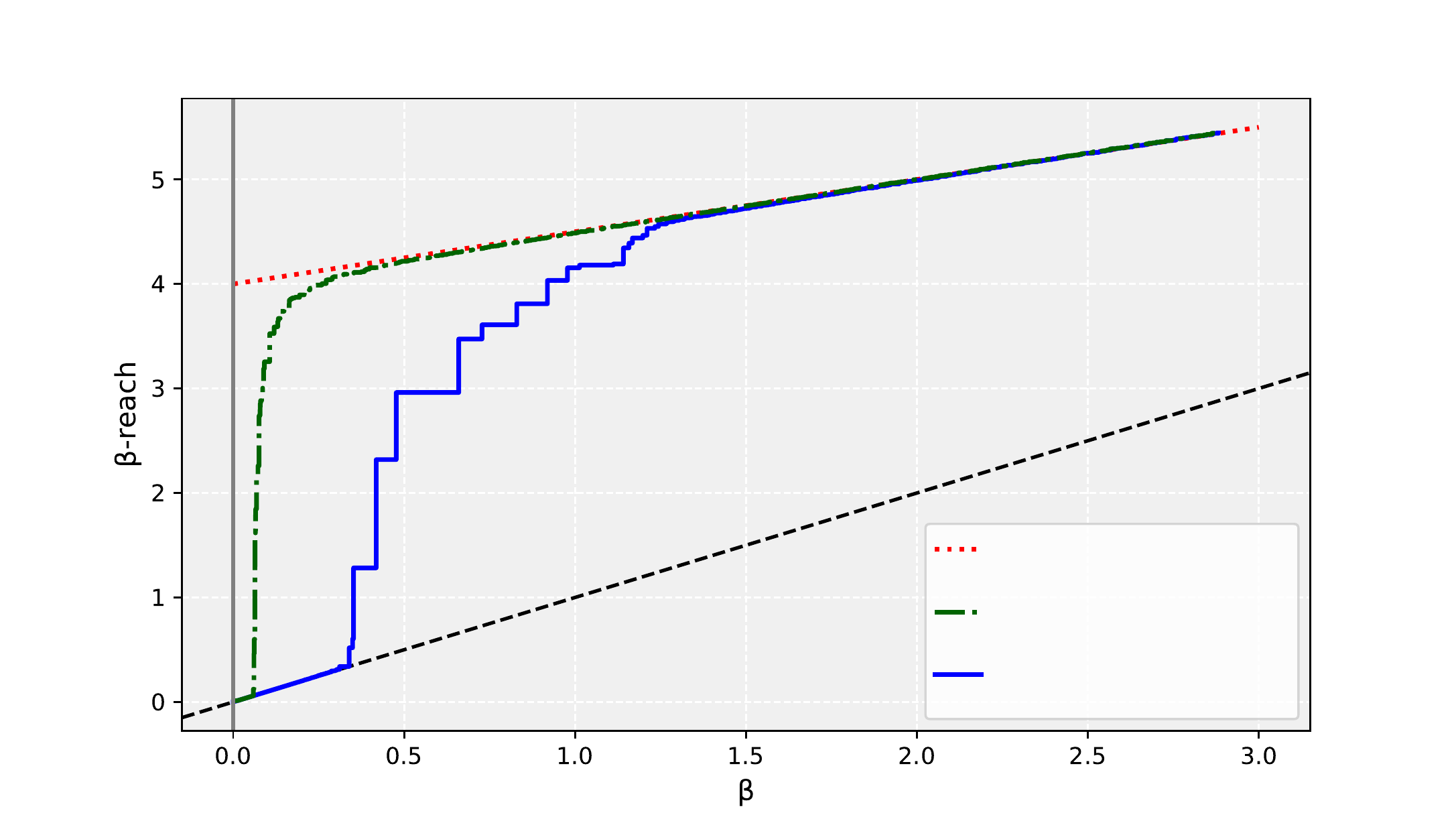}
\put(-94,53){\small{$\reach_\beta(M)$}}
\put(-94,40){\small{$\reach_\beta^{\mathcal{M}}(\hat{M})$}}
\put(-94,27){\small{$\reach_\beta(\hat{M})$}}
\caption{The exact $\beta$-reach profile (see Definition~\ref{def:beta_reach}) of $\hat{M}$ in Example~\ref{exa:point_cloud_surface} and Figure~\ref{fig:point_cloud_surface}~(a) is shown as the blue solid curve. Using the formulation in~\eqref{eqn:mesh_beta_reach}, and the mesh in Figure~\ref{fig:point_cloud_surface}~(b), one may compute an improved approximation (green dashed curve) to the $\beta$-reach of the underlying manifold $M$ (red dotted curve).}
\label{fig:point_cloud_surface_beta_reach}
\end{figure}

\begin{exam}[3-dimensional point cloud on a 2-dimensional manifold]\label{exa:point_cloud_surface}
Let $M\subset\R^3$ be a section of a two-dimensional paraboloid, where its reach is decided by a point of maximal curvature at the vertex ($\reach(M)=4$). Let $\hat{M}\subset M$ be a realization of $1500$ points uniformly distributed on $M$ shown in Figure~\ref{fig:point_cloud_surface}~(a). Using the Python library \texttt{pyvista} \cite{sullivan2019}, we construct a mesh $\mathcal{M}(\hat{M})$ over the point cloud $\hat{M}$ (see Figure~\ref{fig:point_cloud_surface}~(b)). This allows for an improved approximation of the $\beta$-reach profile of $M$,
\begin{align}\label{eqn:mesh_beta_reach}
\reach_\beta^{\mathcal{M}}(\hat{M}) := \inf\bigg\{g_{\norm{a_2-a_1}}(x) :\ 
&a_1,a_2 \in \hat{M},\\
&x=\delta_{\mathcal{M}(\hat{M})}\left(\frac{a_1+a_2}2\right) \geq \beta\bigg\}.\nonumber
\end{align}
The $\beta$-reach of $M$ can be computed exactly as $\reach_\beta(M) = 4 +\frac{\beta}2$, for $\beta$ between 0 and some positive constant. As seen in Figure~\ref{fig:point_cloud_surface_beta_reach}, this is well approximated by $\reach_\beta(\hat{M})$ for $\beta>1.3$, and by $\reach_\beta^{\mathcal{M}}(\hat{M})$ for $\beta > 0.4$; both approximations are efficiently computable from the point cloud $\hat{M}$.
\end{exam}

To estimate $\reach(M)$ from the $\beta$-reach profiles in Example~\ref{exa:point_cloud_surface}, one can perform a linear regression on $\reach_\beta^{\mathcal{M}}(\hat{M})$, for $\beta$ in the range where the curve appears linear, and estimate the reach as the model intercept. 

Although the manifold $M$ in Example~\ref{exa:point_cloud_surface} is derived from a simple paraboloid, this example is representative of other two-dimensional manifolds whose reach is determined by a region of high curvature. The difference might be that the linear term in the first order approximation of the $\beta$-reach profile at $\beta=0$ does not have a coefficent of 1/2 (see Example~\ref{exa:beta_reach_C2}).

The following two examples emphasize the applicability of this method in higher dimensions.

\begin{figure}[t]
\centering
\includegraphics[width=0.8\linewidth]{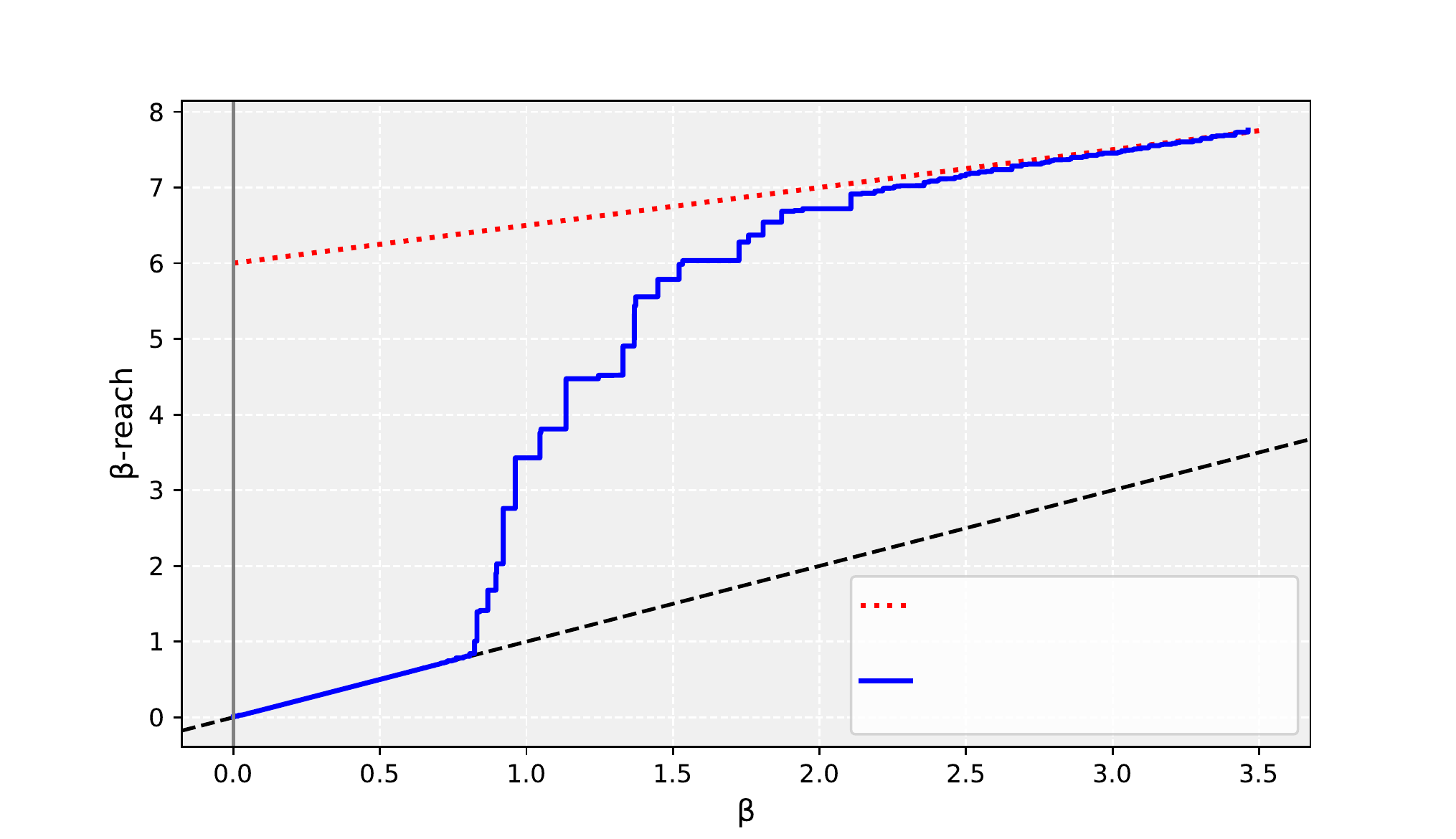}
\put(-93,39){\small{$\reach_\beta(M)$}}
\put(-93,26){\small{$\reach_\beta(\hat{M})$}}
\caption{The exact $\beta$-reach profile of $\hat{M}$ in Example~\ref{exa:d4} is shown as the blue solid curve. The $\beta$-reach of the underlying manifold $M$ (a 3-dimensional paraboloid embedded in $\R^d$) is shown as a red dotted curve.}
\label{fig:d4_beta_reach}
\end{figure}

\begin{exam}[$d$-dimensional point cloud on a 3-dimensional manifold]\label{exa:d4}
Let $d\geq 4$, and let $M\subset\R^d$ be a section of the three-dimensional paraboloid defined by $x_1^2 + x_2^2 + x_3^2 - 12x_4 = 0$, such that its reach is decided by a point of maximal curvature at the origin. It can be shown that $\reach_\beta(M)=6 + \frac{\beta}2$, for $\beta$ between 0 and some positive threshold.

Let $\hat{M}\subset M$ be a realization of $3000$ points, uniformly distributed on $M$. The Python package \texttt{pyvista} does not currently have methods for computing meshes in dimension higher than 3, so in this example, we do not compute $\reach_\beta^{\mathcal{M}}(\hat{M})$ in~\eqref{eqn:mesh_beta_reach}. Mesh generation in higher dimension is itself an active field of research \citep{boissonnat2008, boissonnat2009, edelsbrunner2001}.
Nonetheless, we plot the exact $\beta$-reach profile of $\hat{M}$ in Figure~\ref{fig:d4_beta_reach}.
\end{exam}

\begin{rema}\label{rem:relationship_n_d_m}
The dimension of the ambient space $d\geq 4$ in Example~\ref{exa:d4} does not affect the shape of the $\beta$-reach profile of the point cloud $\hat M$, computed purely in terms of distances between the points, and distances to midpoints. These distances are preserved under isometries to higher dimensional spaces.
However, the number of points in $\hat M$ plays a role in the shape of the $\beta$-reach profile through the Hausdorff distance between $\hat M$ and the underlying manifold $M$. The exact relationship between the number of points in $\hat M$ and the Hausdorff distance $d_H(\hat M, M)$ cannot be made precise with no prior knowledge of how the points are distributed on $M$---which also plays a role in the shape of the $\beta$-reach profile of $\hat M$. Nonetheless, the dimension of $M$ determines to a large extent the number of points needed to ensure that $d_H(\hat M, M)$ is less than a given tolerance.
\end{rema}

\begin{figure}[t]
\centering
\includegraphics[width=0.8\linewidth]{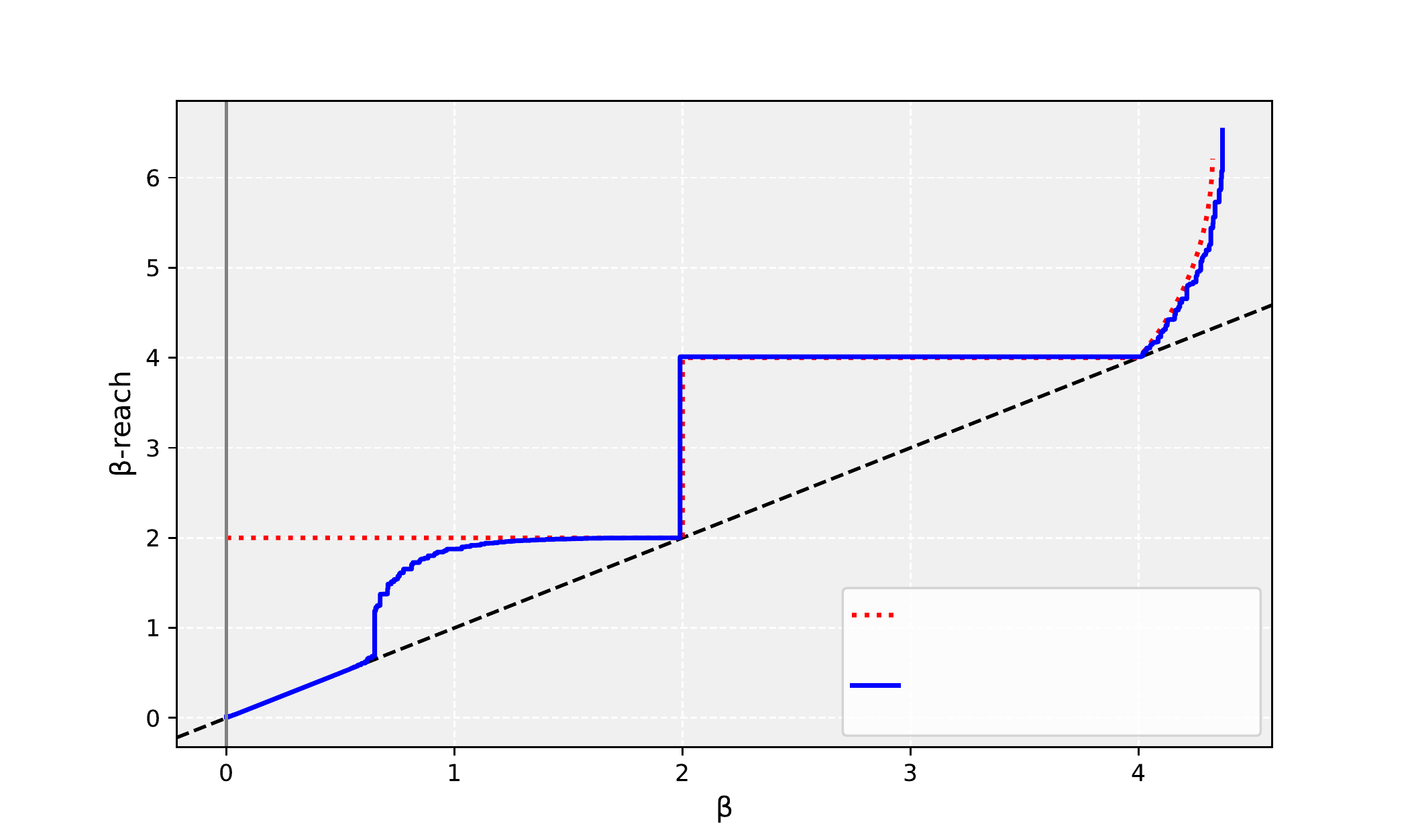}
\put(-93,39){\small{$\reach_\beta(M)$}}
\put(-93,26){\small{$\reach_\beta(\hat{M})$}}
\caption{The exact $\beta$-reach profile of $\hat{M}$ in Example~\ref{exa:spheres}. The $\beta$-reach of the underlying manifold $M$ is shown as a red dotted curve. Here, $M$ is the union of two 3-dimensional hyperspheres of radius 2 embedded in $\R^d$.}
\label{fig:spheres}
\end{figure}

\begin{exam}[Two 3-dimensional hyperspheres]\label{exa:spheres}
Let $d\geq 4$, and let $M\subset\R^d$ be the union of two three-dimensional hyperspheres of radius 2 whose centers are 12 units apart. Let $\hat{M}\subset M$ be a realization of $3000$ points uniformly distributed on $M$. The $\beta$-reach profiles of $M$ and a realization of $\hat{M}$ are shown in Figure~\ref{fig:spheres}.
\end{exam}

Example~\ref{exa:spheres} highlights a few nice features of the $\beta$-reach profile. That is, it adapts very well to situations where the reach is determined by a bottleneck structure, as is the case for a hypersphere. For $\beta\in (2,4]$, the $\beta$-reach is no longer determined by the radius of the hypersphere, but by half the distance between the surfaces of the spheres (in this case, $(12-2-2)/2=4$ units). Thus, the $\beta$-reach profile also provides information about the large scale features of the data. Notably, it gives the scales at which it becomes possible to distinguish these features from one another.

Finally, recall from Remark~\ref{rem:beta_reach_bounded_by_beta} that one can read from the plot information about the critical points of the generalized gradient function. In the case of Example~\ref{exa:spheres}, the distances from the three critical points to $M$ are 2, 2, and 4.

\section{Numerical studies}\label{sec:sim}

In this section, we test the methods for bounding the reach and $r$-convexity introduced in Section~\ref{sec:point_cloud} against numerical data. First, in Section~\ref{sec:sim:aircraft}, we study the performance of our methods on real data. Then, in Section~\ref{sec:sim_convergence}, we test the convergence results in Theorems~\ref{thm:rconv_estimator} and~\ref{thm:reach_estimator}.

\subsection{Aircraft data}\label{sec:sim:aircraft}

The real data in Example~\ref{exa:airplane}, below, is studied using the tools developed in this document. After a short description of the data, we perform several analyses using the tools developed in Section~\ref{sec:point_cloud}. 

\begin{figure}[t]
    \centering
    \begin{subfigure}{0.6\textwidth}
\centering
\includegraphics[width=\textwidth]{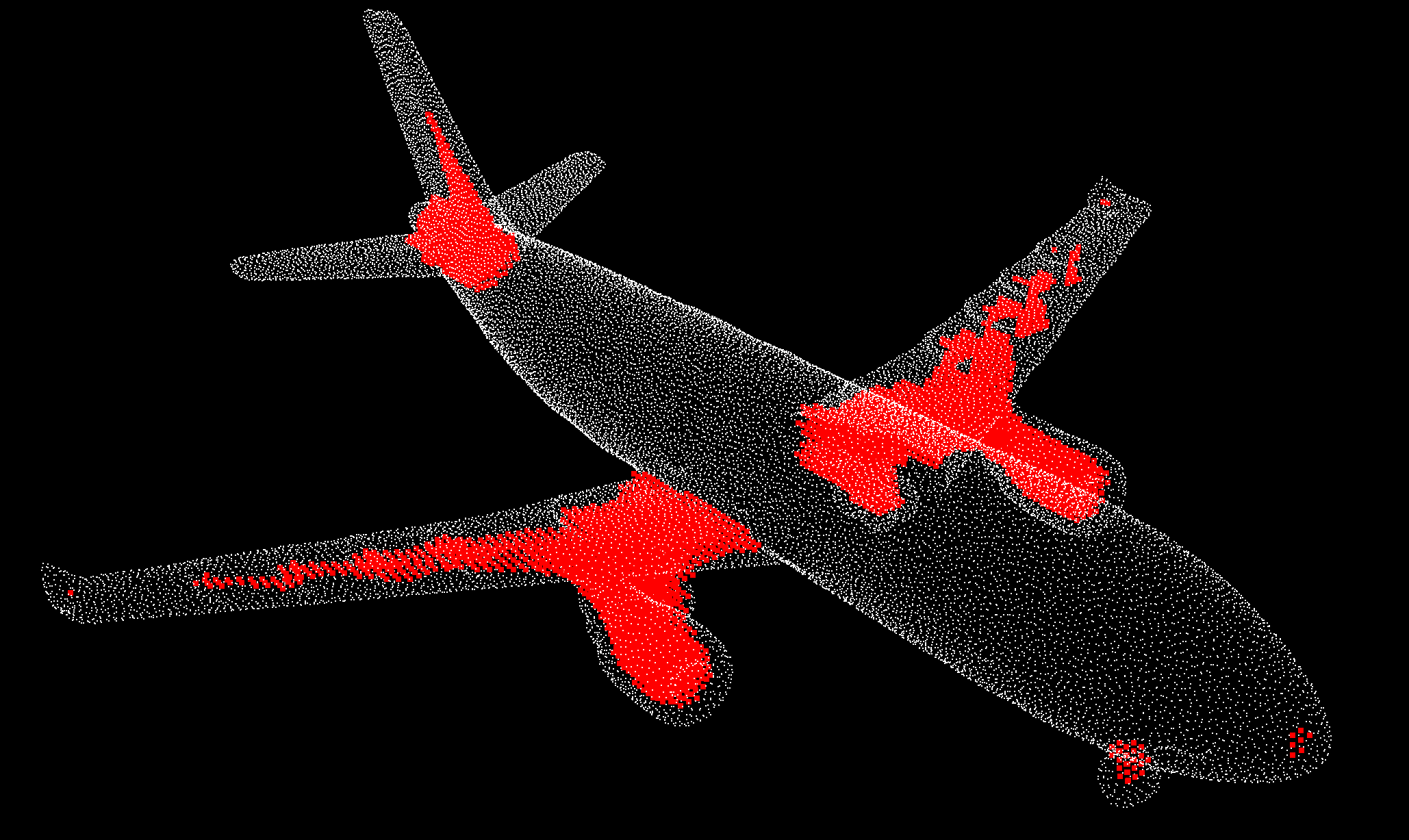}
\caption{}
\label{fig:plane_red}
\end{subfigure}
    \hfill
    \begin{subfigure}{0.3\textwidth}
        \centering
        \includegraphics[width=\textwidth]{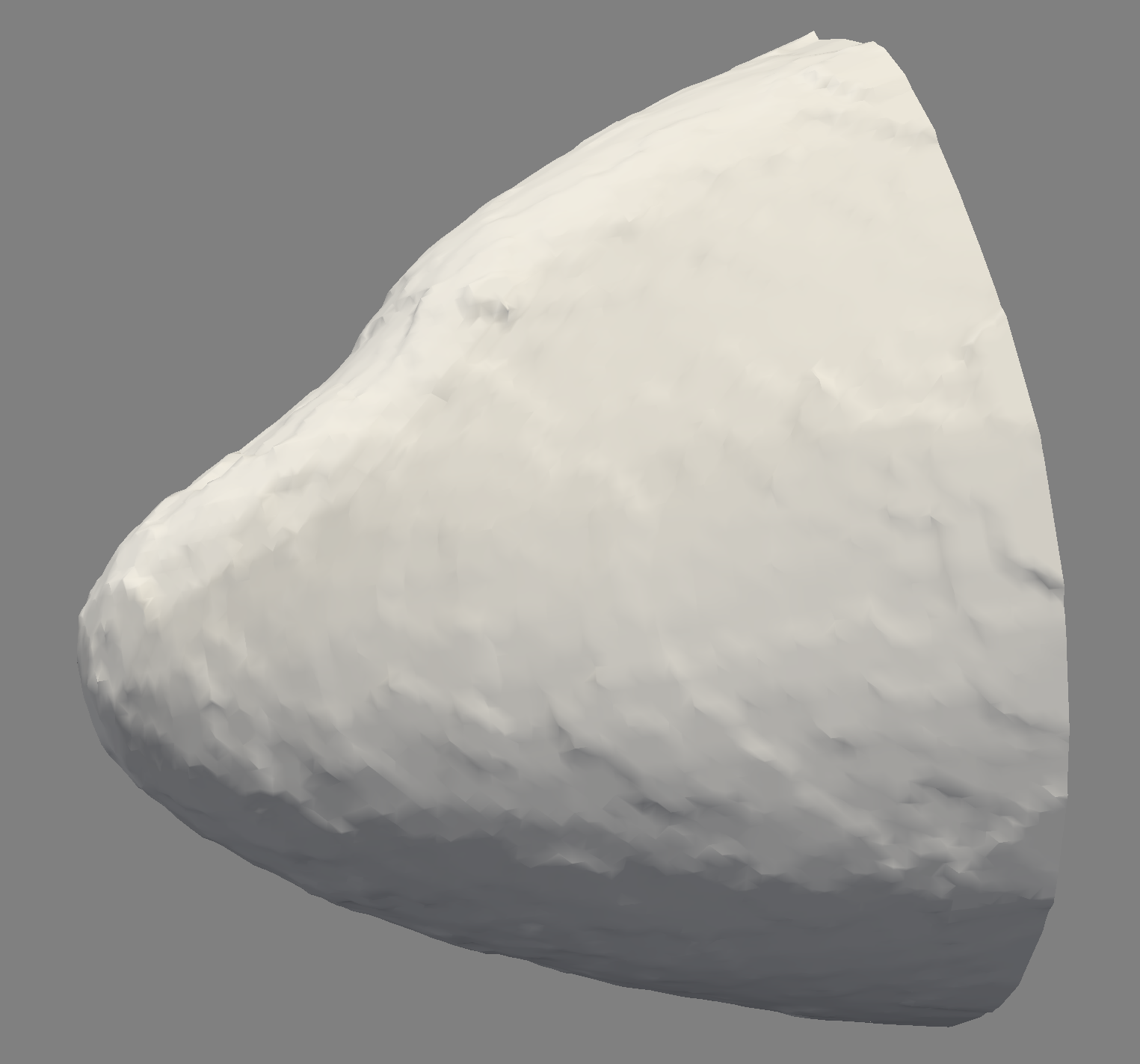}
        \caption{}
        \label{fig:plane_nose}
    \end{subfigure}
    \caption{\textbf{(a):} The point cloud data $\hat{P}$ in Example~\ref{exa:airplane} is used to guarantee that the $r$-convexity of the underlying surface of the plane $P$ does not exceed th chosen test radius $r=0.03$. The areas of the plane responsible for this restriction on the $r$-convexity are highlighted in red. \textbf{(b):} The \texttt{pyvista} mesh $\mathcal{M}(\hat{N})$ constructed from $\hat N\subset \hat P$, the points at the nose of the plane.}
    \label{fig:plane}
\end{figure}

\begin{exam}[Aircraft data]\label{exa:airplane}
\cite{baorui2022} provides a point cloud dataset that is sampled over the surface of a commercial aircraft (the white points in Figure~\ref{fig:plane} (a)). The diameter of the raw point cloud data (which corresponds to the length of the plane) is 0.749 units. Denote the surface of the aircraft by $P$ and the approximating point cloud by $\hat{P}$.

We are also interested in the nose of the aircraft, $N\subset P$, the surface of the first 0.056 length units of the aircraft. A two-dimensional mesh approximating $\hat{N} := N\cap \hat{P}$ is shown in Figure~\ref{fig:plane} (b).
\end{exam}

It is possible that in an engineering practice, one may want to identify which regions of a point cloud are not $r$-convex for a predefined value of $r$. Panel~(a) of Figure~\ref{fig:plane} highlights in red the regions of the plane that are \textit{certainly} not $r$-convex for the choice of $r=0.03$.

The surface of the plane $P$ is known to be two-dimensional, and so a cubic lattice of points $\varphi$ with lattice spacing $a=0.004$ superimposed over $\hat{P}$ will almost surely not intersect $P$. The discrete dilation and erosion operations in Definition~\ref{def:pc} are well defined using the larger point cloud $\pc:= \varphi \cup \hat{P}$.
The red points in Figure~\ref{fig:plane}~(a) are the elements of the discrete set $(\hat{P}_{r - \epsilon})_{-(r+\epsilon)} \cap \varphi$, with $\epsilon:=\sqrt{3}a/2 = \sup\{\delta_{\pc}(q):q\in\R^3\}$. Since this set is not empty, one has that $P$ is a proper subset of $P_{\bullet r}$. In addition, by~\eqref{eqn:reach_inclusion}, one has conclusive evidence that $\reach(P)\leq r$.

This example illustrates that, with this method, one can identify the regions responsible for limiting the $r$-convexity (and thus the reach), with a test specificity of 100\%. One can improve the sensitivity of the test by decreasing the lattice spacing $a$. Then, other regions that are not $r$-convex (such as the interiors of the horizontal and vertical stabilizers) would be identified as such.
\smallskip

Remark that the nose of the plane is marked in red, due of course to a region of high curvature. By considering only the smooth manifold $N$ in Example~\ref{exa:airplane}, corresponding to the nose of the plane, we can approximate the $\beta$-reach profile of $N$ by $\reach_\beta(\hat{N})$ in Definition~\ref{def:beta_reach}, or by $\reach_\beta^{\mathcal M}(\hat{N})$ in Equation~\eqref{eqn:mesh_beta_reach}.

\begin{figure}[t]
\centering
\includegraphics[width=0.8\linewidth]{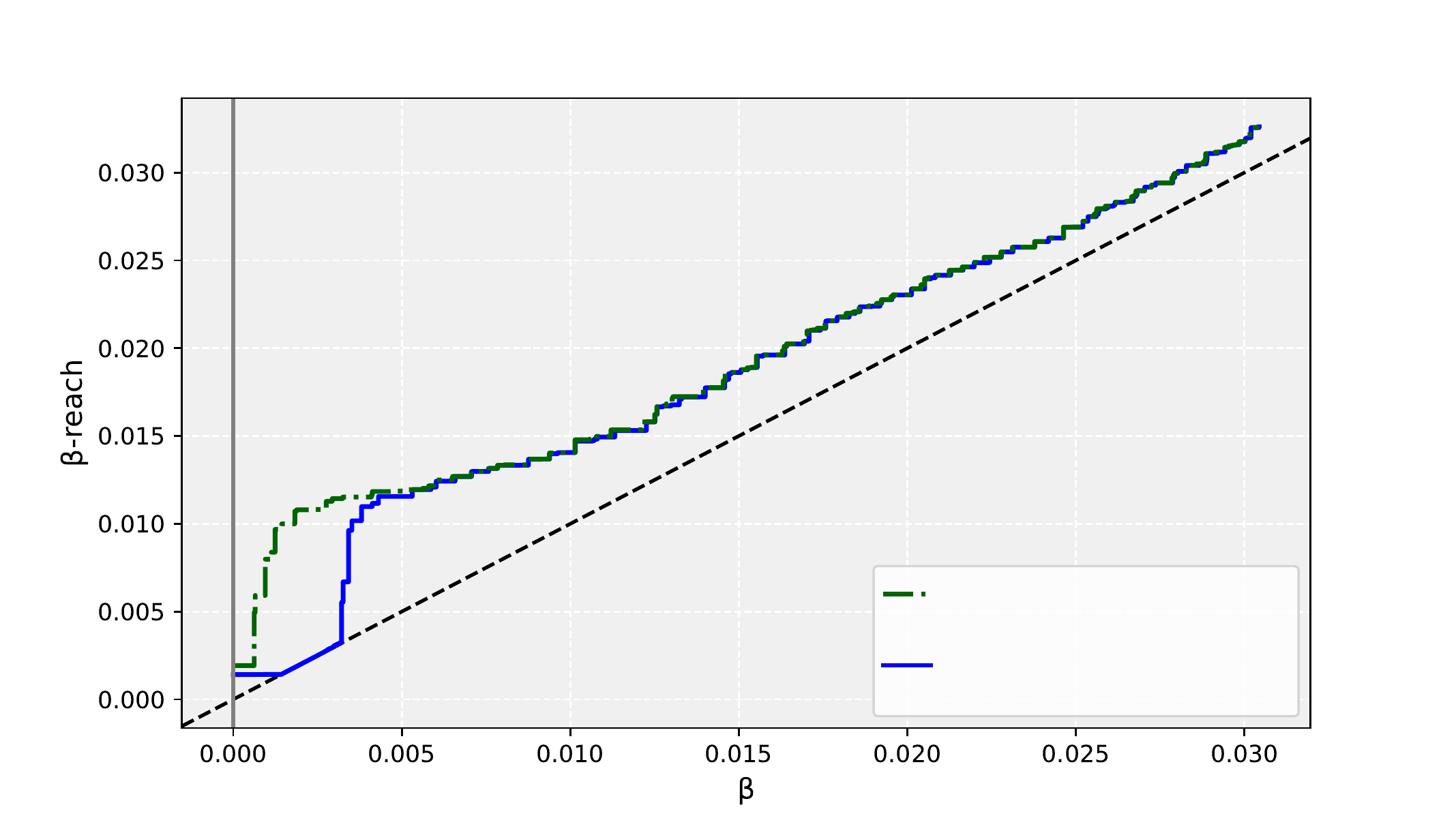}
\put(-93,37){\small{$\reach_\beta^{\mathcal{M}}(\hat{N})$}}
\put(-93,24){\small{$\reach_\beta(\hat{N})$}}
\caption{The exact $\beta$-reach profile of $\hat{N}$ in Example~\ref{exa:airplane} is shown as the blue solid curve. Using the formulation in~\eqref{eqn:mesh_beta_reach}, and the mesh in Figure~\ref{fig:plane}~(b), one may compute an improved approximation (green dashed curve) to the $\beta$-reach of the underlying (unknown) manifold $N$.}
\label{fig:beta_reach_plane}
\end{figure}

The resulting $\beta$-reach profiles are shown in Figure~\ref{fig:beta_reach_plane}. From the figure, the reach is clearly seen to be much smaller than $0.03$ as indicated by the $r$-convexity experiment. Assuming that the $\beta$-reach profile of $N$ maintains a constant slope near 0 as seen before in Figure~\ref{fig:point_cloud_surface_beta_reach}, the approximation $\reach(N)\approx 0.012$ can be read from Figure~\ref{fig:beta_reach_plane}.
\smallskip

%that one should Suppose that one is interested in the ability to discern the homology of the underlying surface 

There are two main issues with using the reach bound in~\eqref{eqn:reach_estimator} on the aircraft data in Example~\ref{exa:airplane}. First, it is expected to produce a large overestimate of $\reach(N)$ since $N$ is a smooth manifold (see the discussion in Section~\ref{sec:point_cloud_beta_reach}). In addition, it is impossible to know the Hausdorff distance beween $N$ and its approximating point cloud $\hat{N}$. Nevertheless, one can obtain a good approximation by considering the persistence diagram of growing balls centered at the points in $\hat{N}$. The largest of the death-times of the topological features corresponds to the smallest radius for which the union of balls is homotopic to a point. This is likely to corespond closely to the Hausdorff distance $d_H(N,\hat{N})$, and is thus a good choice for $\epsilon$ in~\eqref{eqn:reach_estimator}. The persistence diagram, generated from the Python module \texttt{ripser} \cite{bauer2021}, gives that the largest death-time is $\epsilon=0.0065$. From this, one calculates
$$\hat{\mathrm{rch}}^{(\epsilon)} (\hat N)= 0.022.$$
\smallskip

Although the reach bound in~\eqref{eqn:reach_estimator} is less appropriate to use in this setting, we will see in the following section that it (and the $r$-convexity bound in~\eqref{eqn:convexity_estimator}) is highly applicable when studying binary images.

\subsection{Numerical convergence of reach and $r$-convexity bounds}\label{sec:sim_convergence}

In stochastic geometry literature, the \textit{excursion set} of a random field is the subset of its domain on which the random field surpasses a predefined threshold (see \cite{adler2007} for a comprehensive reference). The set in Figure~\ref{fig:corrected_reach_test}, for example, is one realization of the excursion set of a stationary, isotropic Gaussian random field sampled on a square lattice.

The following two examples are meant to imitate the discretized excursion set of $C^2$ continuous random fields on square lattices (see, \textit{e.g.}, \cite{cotsakis2022, cotsakis2022_1, bierme2021}). A key feature of both examples, is that the reach and $r$-convexity of the sets are known, and so we can study the convergence of the bounds in Theorems~\ref{thm:rconv_estimator} and~\ref{thm:reach_estimator} as the grid of sampling points becomes dense in $\R^2$. Each example illustrates one of the two cases mentioned in Remark~\ref{rem:aamari_split} concerning the relationship between the reach, regions of high curvature, and bottleneck structures.

\begin{figure}[t]
\centering
\includegraphics[width=0.66\linewidth]{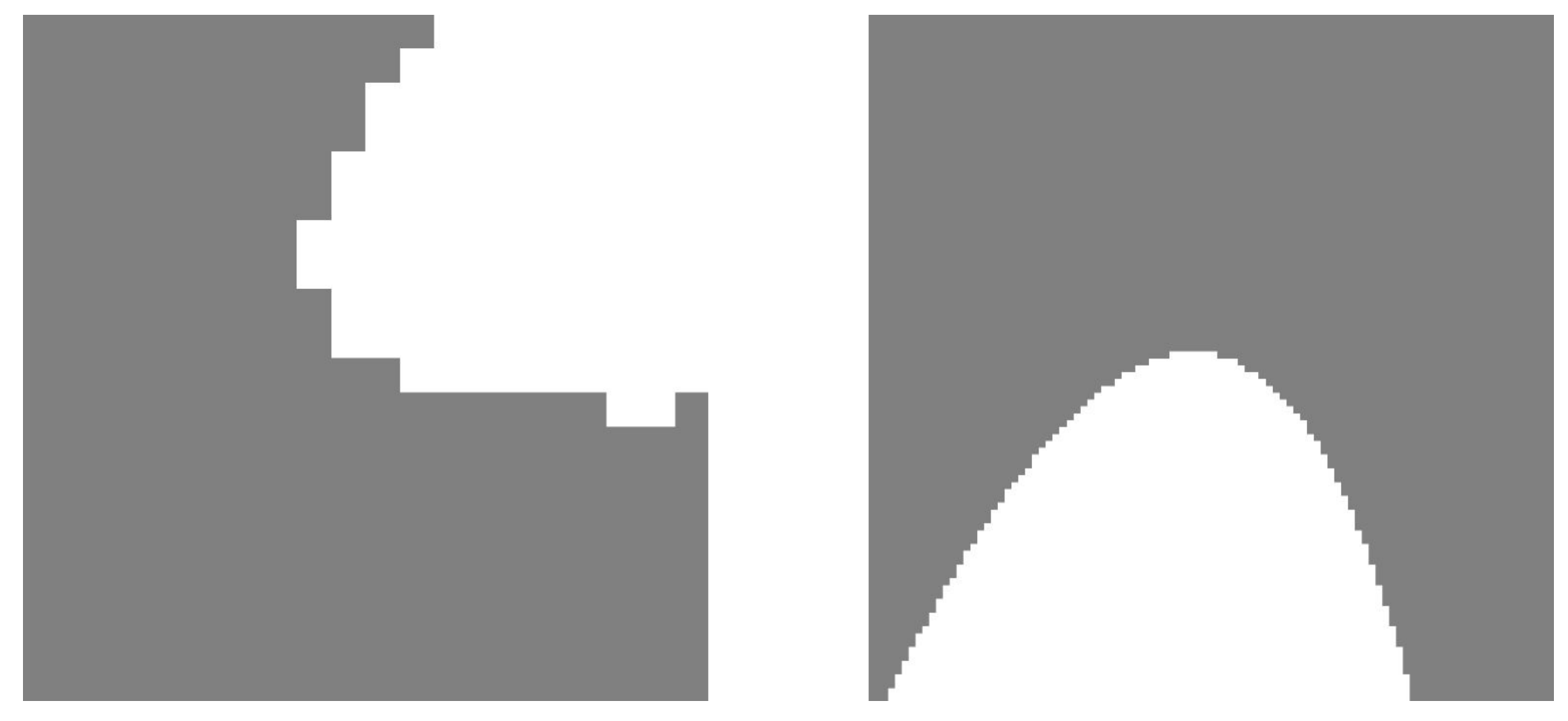}
\put(-58,-10){(b)}
\put(-176,-10){(a)}
\caption{Realizations of the random set $\hat{U}^{(n)}$ in Example~\ref{exa:curvature_sim}. \textbf{(a):} $n=2$. \textbf{(b):} $n=10$.}
\label{fig:curvature_sim}
\end{figure}

\begin{exam}[Reach determined by curvature]\label{exa:curvature_sim}
Let $U := \{(x,y)\in\R^2 : y \leq x^2/2\} \cap B(0,10)$. It is easy to check that $\reach(U) = \rconv(U) = 1$. Moreover, the reach is determined by a point of maximal curvature at $(0,0)\in U$.

Let $(\pc^{(n)})_{n\geq 1}$ be a sequence of square lattices over $\R^2$ with lattice spacing $a_n = 0.7/n$, and independent, uniformly random position and orientation. For $n\in\N^+$, let $\hat{U}^{(n)}:= U \cap \pc^{(n)}$ be the set $U$ sampled on $\pc^{(n)}$. Figure~\ref{fig:curvature_sim} depicts square subsets of realizations of $\hat{U}^{(2)}$ and $\hat{U}^{(10)}$ shown as binary images.
\end{exam}

\begin{figure}[t]
\centering
\includegraphics[width=0.66\linewidth]{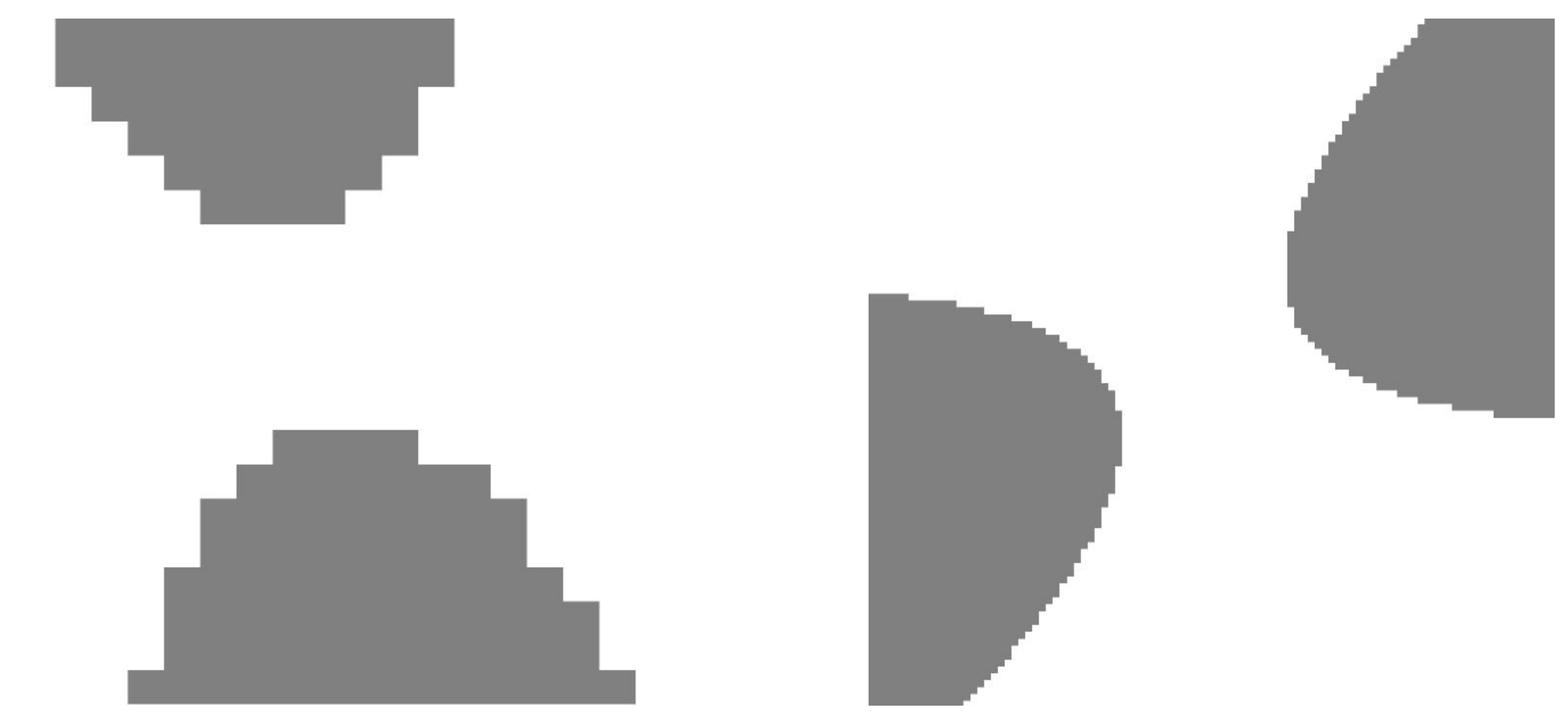}
\put(-58,-10){(b)}
\put(-176,-10){(a)}
\caption{Realizations of the random set $\hat{W}^{(n)}$ in Example~\ref{exa:bottleneck_sim}. \textbf{(a):} $n=2$. \textbf{(b):} $n=10$.}
\label{fig:bottleneck_sim}
\end{figure}

\begin{exam}[Reach determined by a bottleneck structure]\label{exa:bottleneck_sim}
Let $W:= \{(x,y)\in\R^2 : \vert y\vert \geq x^2/2 + 1\}\cap B(0,10)$. One has $\reach(W)=\rconv(W)=1$, which is half the distance between its two connected components.

For the sequence of square lattices $(\pc^{(n)})_{n\geq 1}$ in Example~\ref{exa:curvature_sim}, let $\hat{W}^{(n)} := W\cap \pc^{(n)}$. Figure~\ref{fig:bottleneck_sim} depicts square subsets of realizations of $\hat{W}^{(2)}$ and $\hat{W}^{(10)}$ shown as binary images.
\end{exam}

\begin{rema}
The choice to use sets with a reach of 1 in Examples~\ref{exa:curvature_sim} and~\ref{exa:bottleneck_sim} does not limit their generality. The importance lies in the scale of the sampling lattice $\pc^{(n)}$ relative to the reach, and the ratio of the two tends to 0 as $n\to\infty$ in both examples.

The following observation makes the gernerality of these examples even clearer: to consider the lattice of sampling points at various scales and orientations is equivalent to considering a fixed lattice and various scales and orientations of the underlying sets $U$ and $W$.
\end{rema}

\begin{figure}[t]
\centering
\includegraphics[width=\linewidth]{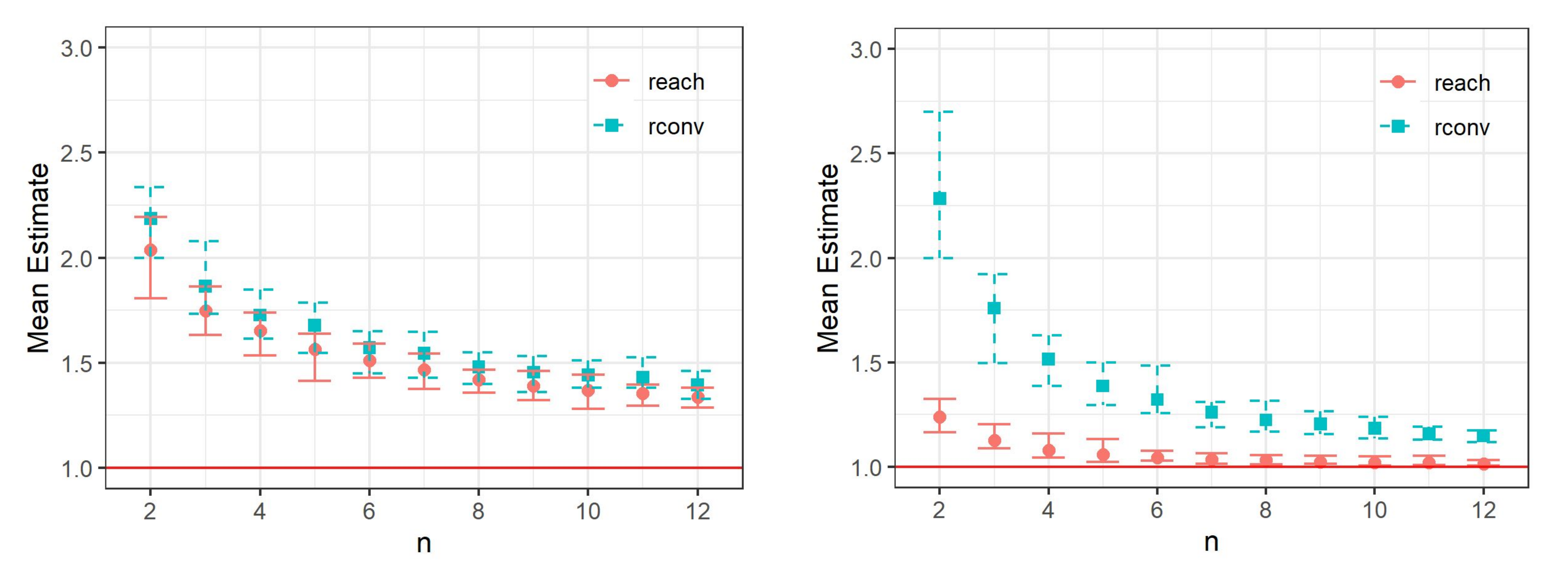}
\put(-83,-10){(b)}
\put(-253,-10){(a)}
\caption{\textbf{(a):} The mean and 95\% confidence interval for the values of $\hat{\mathrm{r}}^{(a_n/\sqrt{2})}(\hat{U}^{(n)})$ (squares) and $\hat{\mathrm{rch}}^{(\sqrt{1.25}a_n)}(\hat{U}^{(n)})$ (circles) for 50 independent realizations of $\hat{U}^{(n)}$ in Example~\ref{exa:curvature_sim} \textbf{(b):} The mean and 95\% confidence interval for the values of $\hat{\mathrm{r}}^{(a_n/\sqrt{2})}(\hat{W}^{(n)})$ (squares) and $\hat{\mathrm{rch}}^{(\sqrt{1.25}a_n)}(\hat{W}^{(n)})$ (circles) for 50 independent realizations of $\hat{W}^{(n)}$ in Example~\ref{exa:bottleneck_sim}.}
\label{fig:plots}
\end{figure}

For each $n\in\{2,\ldots,12\}$, we compute 50 independent realizations of $\hat{U}^{(n)}$ and 50 independent realizations of $\hat{W}^{(n)}$. For each of the resulting discrete sets, we compute the bounds of $r$-convexity and reach in~\eqref{eqn:convexity_estimator} and~\eqref{eqn:reach_estimator} respectively.
The means of the bounds are plotted in Figure~\ref{fig:plots} with empirical 95\% confidence intervals. Panels~(a) and~(b) correspond to the results for the replications of $\hat{U}^{(n)}$ and $\hat{W}^{(n)}$ respectively.

\begin{figure}[t]
\centering
\includegraphics[width=0.55\linewidth]{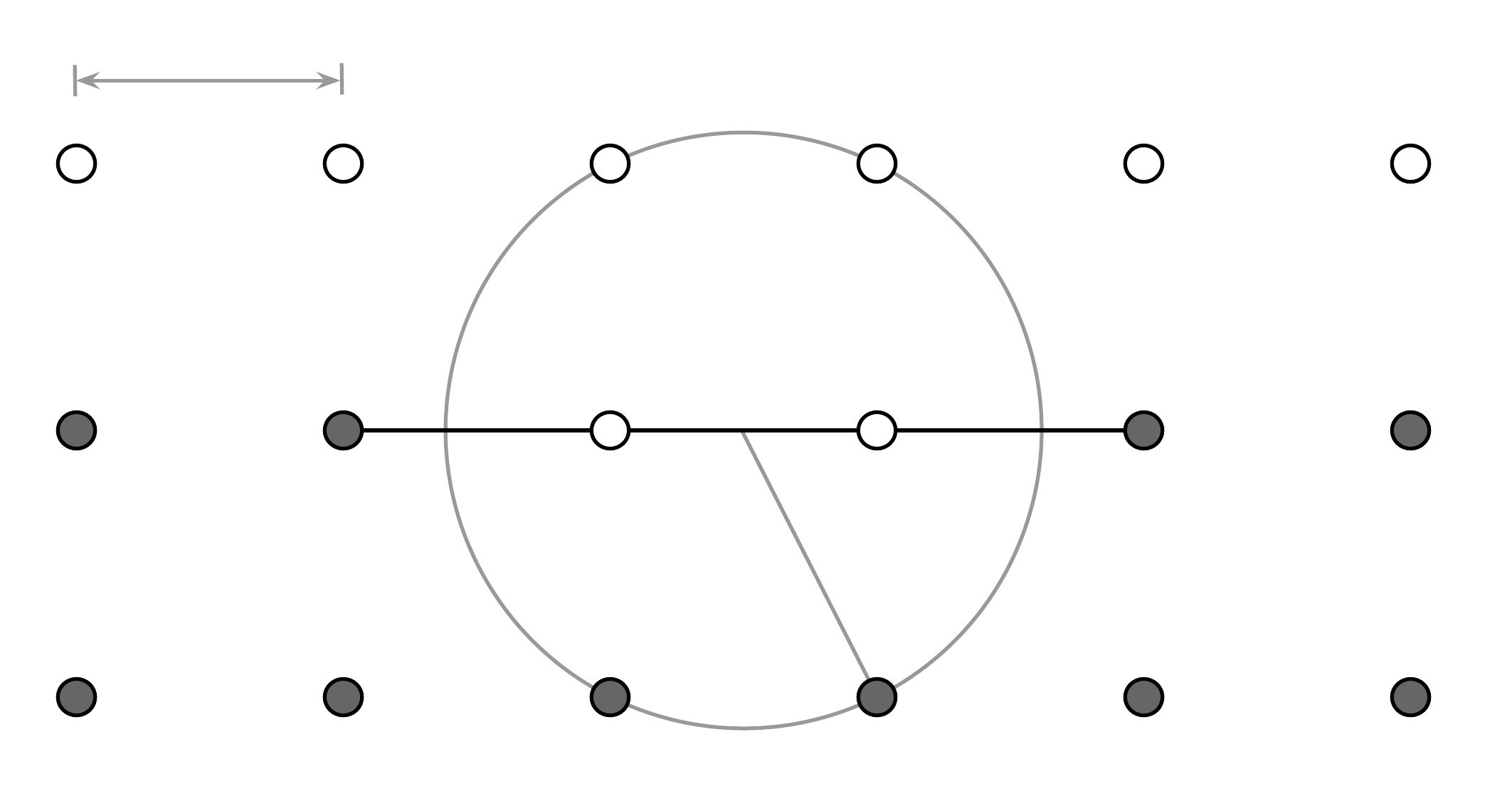}
\put(-167,77){\small{$a_n$}}
\put(-120,24){\small{$\sqrt{1.25}a_n$}}
\caption{The grey points and the white points can be seperated by a curve with arbitrarily small curvature (arbitrarily large reach). The midpoint of two grey points may be up to a distance of $\sqrt{1.25}a_n$ from another grey point, while remaining arbitrarily close to the separating curve.}
\label{fig:explanation125}
\end{figure}

For the $r$-convexity bound, $\epsilon_n = a_n/\sqrt{2}$ is chosen since $d_H(\pc^{(n)},\R^2) = a_n/\sqrt{2}$, for $n\in\N^+$.

For the reach bound, we set $\epsilon_n = \sqrt{1.25}a_n$ for reasons that are more complicated. Even though this choice of $\epsilon_n$ leads to $d_H(\hat{U}^{(n)},U) > \epsilon_n$, or $d_H(\hat{W}^{(n)},W) > \epsilon_n$ with positive probability, it is sufficient for the result of Theorem~\ref{thm:reach_estimator} in this case. We provide Figure~\ref{fig:explanation125} for some intuition, but a formal proof is omitted.
\smallskip

A linear regression on the log-log plot of the sample means of the data in Figure~\ref{fig:plots}~(a) shows that
$$\E[\hat{\mathrm{r}}^{(a_n/\sqrt{2})}(\hat{U}^{(n)})] \approx 1 + 1.71n^{-0.59}\ \ \mbox{and}\ \ \E[\hat{\mathrm{rch}}^{(\sqrt{1.25}a_n)}(\hat{U}^{(n)})] \approx 1 + 1.54n^{-0.62}.$$
These empirical convergence rates --- $\bigo(\epsilon_n^{0.59})$ for the bound on $\rconv(U)$ and $\bigo(\epsilon_n^{0.62})$ for the bound on $\reach(U)$ --- are both slightly faster than the predicted rate of $\bigo(\sqrt{\epsilon_n})$ established in~\eqref{eqn:rch_converge_rate} for the bound on the reach. Since the reach of $U$ is determined by a region of maximal curvature, the convergence rate is governed by the analysis in Case~2 in the proof of Theorem~\ref{thm:reach_estimator}.

For the data in panel~(b) of Figure~\ref{fig:plots},
$$\E[\hat{\mathrm{r}}^{(a_n/\sqrt{2})}(\hat{W}^{(n)})] \approx 1 + 2.81n^{-1.20}\ \ \mbox{and}\ \ \E[\hat{\mathrm{rch}}^{(\sqrt{1.25}a_n)}(\hat{W}^{(n)})] \approx 1 + 0.65n^{-1.50}.$$
Again, these empirical convergence rates are faster than the anticipated rate of $\bigo(\epsilon_n)$ established for the bound on the reach in Case~1 (\textit{i.e.}, when the reach is determined by a bottleneck structure) in the proof of Theorem~\ref{thm:reach_estimator}.

\section{Proofs and technical results}\label{sec:proofs}

Here, we provide proofs and auxiliary lemmas that support the results in Sections~\ref{sec:definitions} and~\ref{sec:point_cloud}.

%\begin{lemm}\label{lem:finite_cardinality}
%Let $r\in\R^+$ and let $A\subset\R^d$ be a closed set satisfying $\rconv(A) \neq r$. For all $p\in A_{\bullet r}$, there exists a subset of $A$ with finite cardinality belonging to $\bpsi(p,r)$.
%\end{lemm}

\begin{figure}[t]
\centering
\includegraphics[width=0.6\linewidth]{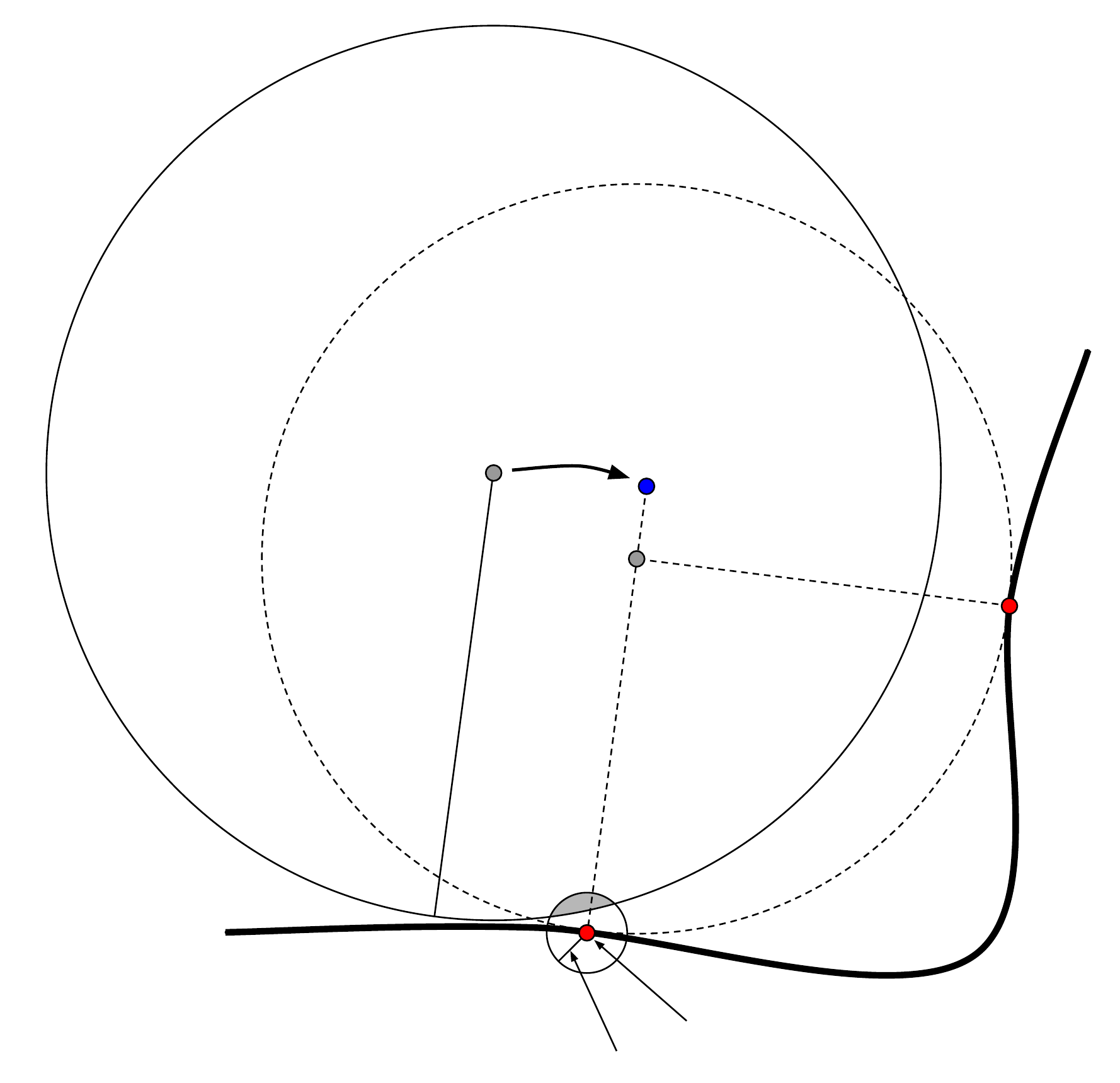}
\put(-3,125){\small{$\partial A$}}
\put(-90,-1){\small{$\epsilon$}}
\put(-80,5){\small{$a_1$}}
\put(-17,83){\small{$a_2$}}
\put(-84,111){\small{$a_1 + rn_1$}}
\put(-85,83){\small{$p$}}
\put(-120,115){\small{$x$}}
\put(-126,69){\small{$r$}}
\caption{The quantities used in the proof of Theorem~\ref{thm:reach_closing}. Sending $\epsilon\to 0$ leads to the largest circle intersecting $\partial A$ if it is to always intersect $B(a_1,\epsilon)$.}
\label{fig:proof_helper_technical}
\end{figure}
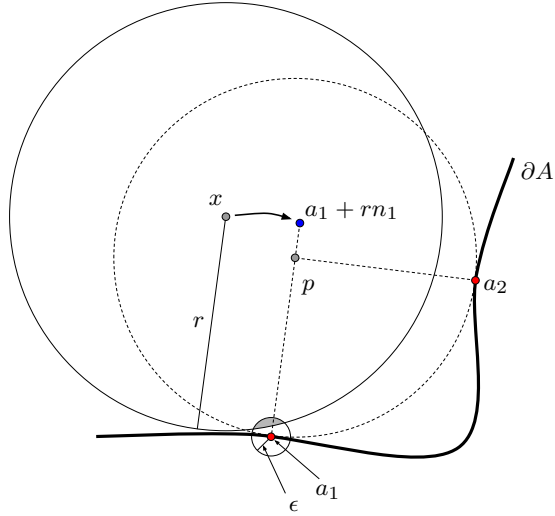

\begin{proof}[Proof of Theorem~\ref{thm:reach_closing}]
We begin by showing~\eqref{eqn:reach_inclusion}. Suppose that $r \in (0,\reach(A))$. By Lemma~\ref{lem:additivity_of_dilation}, $A \subseteq A_{\bullet r}$, so it suffices to show that $A^c \subseteq (A_{\bullet r})^c$. Let $x\in A^c$. If $x\in (A_r)^c$, then clearly, $x\in (A_r)^c \oplus B(0,r) = (A_{\bullet r})^c$. If $x\in A_r\setminus A$, then by \citet[Corollary 4.9]{federer1959},
$$\delta_{(A_r)^c}(x) = r - \delta_A(x) < r,$$
and so $x \in (A_r)^c\oplus B(0,r) = (A_{\bullet r})^c$, which proves~\eqref{eqn:reach_inclusion}. What remains to be shown is that if $\partial A$ is $C^1$-smooth and $(d-1)$-dimensional, then $\rconv(A) \leq \reach(A)$. This inequality is shown via proof by contradiction. Suppose that
\begin{enumerate}
\item[(i)] $r > \reach(A)$\qquad and
\item[(ii)] $A_{\bullet r} = A$.
\end{enumerate}
By~(i), there exists $p\in A_{\frac{r + \reach(A)}{2}}$ with no unique point in $A$ closest to $p$. In particular, since $A$ is closed, there exists two non-identical points $a_1, a_2\in \partial A \subseteq A$ such that $\norm{p-a_1} = \norm{p-a_2} = \delta_A(p) < r$. Let $n_1$ be the unit normal vector to $\partial A$ at $a_1$, pointing towards $p$. Since $a_1$ is a limit point of $A$, one has $A^c \cap B(a_1,\epsilon) \neq \emptyset$ for all $\epsilon> 0$. By~(ii), $A^c = \bigcup_{x\in (A_r)^c} B(x,r)$, and so for all $\epsilon > 0$, there exists $x\in (A_r)^c$ such that
\begin{enumerate}
\item[(iii)] $B(x,r)\cap B(a_1,\epsilon) \neq \emptyset$\qquad and
\item[(iv)] $B(x,r)\cap A = \emptyset$.
\end{enumerate}
The boundary $\partial A$ is $C^1$-smooth at $a_1$, so it is easily checked that for $\epsilon$ close to 0, the locations $x$ that satisfy both~(iii) and~(iv) are necessarily contained in a small neighbourhood around $a_1 + rn_1$ (see Figure~\ref{fig:proof_helper_technical}). That is, there exists a mapping $\theta:\epsilon\mapsto \theta(\epsilon)\in\R^+$ that tends to 0 as $\epsilon\to 0$, such that if $x\in(A_r)^c$ and $\epsilon\in\R^+$ satisfy~(iii) and~(iv), then $x\in B(a_1+rn_1,\theta(\epsilon))$. Note that $\norm{(a_1 + rn_1) - a_2} < r$ since $a_2\in B(p,\delta_A(p)) \subset B(a_1 + rn_1, r)$. Choose $\epsilon$ such that $\theta(\epsilon) < r - \norm{(a_1 + rn_1) - a_2}$. Then for any $x\in(A_r)^c\cap B(a_1+rn_1,\theta(\epsilon))$,
\begin{align*}
\norm{x-a_2} &\leq \norm{x-(a_1+rn_1)} + \norm{(a_1 + rn_1) - a_2}\\
&\leq \theta(\epsilon) + \norm{(a_1 + rn_1) - a_2} < r.
\end{align*}
This contradicts~(iv).
\end{proof}

\begin{lemm}\label{lem:bullet_idempotence}
Let $U$ be an open set in $\R^d$, and let $r \geq 0$. Then,
$$\cl(U)_{\bullet r} = \cl(U)\qquad \Longrightarrow\qquad U_{\bullet r}= U.$$
\end{lemm}
\begin{proof}[Proof of Lemma~\ref{lem:bullet_idempotence}]
By Lemma~\ref{lem:additivity_of_dilation}, item~(b), $U\subseteq U_{\bullet r} \subseteq \cl(U)_{\bullet r}$. Note also that $U_{\bullet r}$ is open. Suppose that $\cl(U)_{\bullet r} = \cl(U)$. Then $U_{\bullet r}$ is an open subset of $\cl(U)$ that contains all of the interior points of $\cl(U)$. Therefore, $U_{\bullet r} = U$.
\end{proof}

\begin{lemm}\label{lem:closed_complement}
Let $A\subseteq \R^d$ be closed. It holds that
$$\rconv(\cl(A^c)) = \sup\{r\in\R : A_{\bullet -r} = A\}$$
\end{lemm}
\begin{proof}[Proof of Lemma~\ref{lem:closed_complement}]
By Lemma~\ref{lem:bullet_idempotence} and item~(a) in Lemma~\ref{lem:additivity_of_dilation}, $\{r\in\R : \cl(A^c)_{\bullet r} = \cl(A^c)\} \subseteq \{r\in\R : (A^c)_{\bullet r} = A^c\} = \{r\in\R : A_{\bullet -r} = A\}$. Therefore, $\rconv(\cl(A^c)) \leq \sup\{r\in\R : A_{\bullet -r} = A\}$. Now we show the reverse inequality. Let $r,\tilde r \in \R^+$ satisfy $r> \tilde r > \rconv(\cl(A^c))$. By Proposition~\ref{prp:open_subset_rconv}, $\cl(A^c)_{\bullet \tilde r} \setminus \cl(A^c)$ contains a ball of radius $\epsilon$ for some $\epsilon \in (0,r)$. Now, let $\delta = \min(\epsilon, 2(r-\tilde{r}))$. Note that $(A^c)_{\delta/2}\supseteq \cl(A^c)$, and so by Lemma~\ref{lem:additivity_of_dilation}, $((A^c)_{\delta/2})_{\bullet r-\delta/2} = ((A^c)_r)_{-(r-\delta/2)} \supseteq \cl(A^c)_{\bullet r-\delta/2} \supseteq \cl(A^c)_{\bullet \tilde r}$, which contains the ball of radius $\epsilon$. Now, eroding by $\delta/2$ preserves a ball of radius $\epsilon - \delta/2 \geq \epsilon/2$. That is, $(((A^c)_{\delta/2})_{\bullet r-\delta/2})_{-\delta/2} = (A^c)_{\bullet r}$ contains a ball of radius $\epsilon/2$ that does not intersect $\cl(A^c)$, and so $(A^c)_{\bullet r} \neq A^c$ which implies the desired $A_{\bullet -r} \neq A$ by Lemma~\ref{lem:additivity_of_dilation}, item~(a).
\end{proof}

\begin{proof}[Proof of Corollary~\ref{cor:serra}]
If $\rconv(A)=0$, then by Theorem~\ref{thm:reach_closing}, $\reach(A)=0$ and we are done. Now assume $\rconv(A) > 0$ and $\rconv(\cl(A^c)) > 0$, then by Lemma~\ref{lem:closed_complement}, there exists $\delta > 0$ such that $A_{\bullet r}= A$ for all $r\in (-\delta,\delta)$. Theorem~1 in \cite{walther1999} states that $\partial A$ is $(d-1)$-dimensional and $C^1$-smooth, which implies Equation~\ref{eqn:reach_equality} in our Theorem~\ref{thm:reach_closing}.
\end{proof}

\begin{proof}[Proof of Corollary~\ref{cor:small_dilation}]
By Theorem~\ref{thm:reach_closing}, the statement holds if $\partial (A_\epsilon)$ is $C^1$-smooth and $(d-1)$-dimensional. Otherwise, if there is indeed a point $c$ on the boundary of $A_\epsilon=\bigcup_{a\in A} B(a,\epsilon)$ that does not have a continuous derivative, then $c$ is an intersection point of two distinct $d$-spheres centered in $A$ of radius $\epsilon$. In other words, $c$ is at cusp that points inwards towards the interior of $A_\epsilon$, making $\reach(A_\epsilon) = \rconv(A_\epsilon)=0$.
\end{proof}

\begin{rema}\label{rem:proof_of_coro_2}
Remark that $(A_{\epsilon})_{\bullet -\epsilon} = A_\epsilon$. Therefore, if the hypotheses of Corollary~\ref{cor:small_dilation} are strengthened to those of Corollary~\ref{cor:serra}, then the proof of Corollary~\ref{cor:small_dilation} holds by applying Theorem~1 in \cite{walther1999} followed by our Theorem~\ref{thm:reach_closing}.
\end{rema}

Here, we present an auxiliary lemma, and use it in our proof of Proposition~\ref{prp:open_subset_rconv} in Section~\ref{sec:point_cloud_rconv}.

\begin{lemm}\label{lem:psi_equivalence}
For closed $A\subseteq \R^d$ and $r\in\R^+$, it holds that
\begin{equation}\label{eqn:psi_equivalence}
A_{\bullet r} = \left\{p\in \R^d : \forall x\in  B(0,r),\,B(p+x,r)\cap A \neq \emptyset\right\}.
\end{equation}
\end{lemm}
\begin{proof}[Proof of Lemma~\ref{lem:psi_equivalence}]
By manipulating the expressions in Definition~\ref{def:minkowski}, one obtains
\begin{equation}\label{eqn:complement_of_balls}
A_{\bullet r} = \left(\bigcup_{y\in(A\oplus B(0,r))^c} B(y,r)\right)^c,
\end{equation}
which reads: \emph{the elements of $A_{\bullet r}$ are the $p\in\R^d$ such that $p$ is not contained in any closed ball of radius $r$ that does not intersect $A$}. This statement is equivalent to: \emph{$p\notin A_{\bullet r}$ if and only if there exists a closed ball of radius $r$ that contains $p$ but does not intersect $A$.} Thus, we have shown that $(A_{\bullet r})^c$ is equal to the complement of the RHS of~\eqref{eqn:psi_equivalence}, which proves the desired result.
\end{proof}

\begin{proof}[Proof of Proposition~\ref{prp:open_subset_rconv}]
Let $\tilde r \in (\rconv(A), r)$ and fix $p\in A_{\bullet \tilde r}\setminus A$. Since $A$ is closed, there is an open neighbourhood containing $p$ that does not intersect $A$. Choose $\epsilon \in \R^+$ such that $B(p,\epsilon)\cap A = \emptyset$. There exists a sufficiently small $\delta \in (0,\epsilon)$ such that for every $y\in B(p,\delta + r)$, there exists an $x\in B(0,\tilde r)$ that satisfies $B(p+x,\tilde r) \setminus B(p,\epsilon) \subset B(y,r)$. Since $p\in A_{\bullet \tilde r}$, one has that $B(p+x,\tilde r) \setminus B(p,\epsilon)$ contains an element of $A$ by Lemma~\ref{lem:psi_equivalence}. By inclusion, $B(y,r)$ contains a point in $A$ as well.

Let $z\in B(p,\delta)$. We have shown that $z$ is in the right-hand side of~\eqref{eqn:psi_equivalence}, since, for all $x\in B(0,r)$, one has $y:=z+x\in B(z,\delta + r)$ by the triangle inequality, and so by previous arguments, $B(z+x,r)$ contains an element of $A$. Therefore, by Lemma~\ref{lem:psi_equivalence}, $z\in A_{\bullet r}$. But $z$ is not in $A$ since $B(p,\delta) \subseteq B(p,\epsilon) \subseteq A^c$. Therefore, $z\in A_{\bullet r}\setminus A$ and so $B(p,\delta) \subseteq A_{\bullet r}\setminus A$.
\end{proof}

\begin{proof}[Proof of Theorem~\ref{thm:rconv_estimator}]
We begin by showing~\eqref{eqn:rconv_bound}.
Let $n\in\N^+$ and fix $p\in\pc^{(n)}\setminus \hat{A}^{(n)}$. If $\epsilon_n\geq \rconv(A)$, then~\eqref{eqn:rconv_bound} holds trivially. Now, let $r\in\R^+$ be such that $\epsilon_n < r < \rconv(A)$ so that $A_{\bullet r} = A$. We aim to show that $p\notin (\hat{A}^{(n)}_{r-\epsilon_n})_{-(r+\epsilon_n)}$. Indeed, by the $r$-convexity of $A$, there exists $x\in (A_r)^c$ such that $\norm{x-p} < r$. In addition, $B(x,\epsilon_n)\cap A_{r-\epsilon_n} = \emptyset$, so there exists $q\in\pc^{(n)}\setminus \hat{A}^{(n)}_{r-\epsilon_n}$ such that $\norm{q-x} \leq \epsilon_n$. Therefore, by the triangle inequality, $\norm{q-p} \leq r + \epsilon_n$, which implies $\delta_{\pc^{(n)}\setminus \hat{A}^{(n)}_{r-\epsilon_n}}(p) \leq r + \epsilon_n$. Thus, $p\notin (\hat{A}^{(n)}_{r-\epsilon_n})_{-(r+\epsilon_n)}$ as desired.
\smallskip

Now, to prove~\eqref{eqn:rconv_lim}, first fix $r > \rconv(A)$. To simplify notation, let $\delta_n := d_H(\hat{A}^{(n)},A)$. By Proposition~\ref{prp:open_subset_rconv}, there exists an open subset $O \subseteq A_{\bullet r}\setminus A$ that contains a closed ball of radius $\delta_{n_0} + 3\epsilon_{n_0}$ for sufficiently large $n_{0}\in\N^+$. Let $n\geq n_0$. The point cloud $\pc^{(n)}$ is sufficiently dense such that there exists $q \in O_{-(\delta_n + 2\epsilon_n)}\cap \pc^{(n)}$. 
Importantly, this implies that $q\in (A_r)_{-(r+\delta_n + 2\epsilon_n)}\cap \pc^{(n)}$. By Lemma~\ref{lem:additivity_of_dilation}, for all $s > r$, we have $(A_r)_{-(r+\delta_n + 2\epsilon_n)} \subseteq ((A_r)_{\bullet (s-r)})_{-(r+\delta_n + 2\epsilon_n)} = (A_s)_{-(s+\delta_n + 2\epsilon_n)}$, and therefore $q\in (A_s)_{-(s+\delta_n + 2\epsilon_n)}\cap \pc^{(n)}$. Notice that $A_s \cap \pc^{(n)} \subseteq \hat{A}^{(n)}_{\delta_n+s}$, which implies $(A_s)^c \supseteq \pc^{(n)}\setminus \hat{A}^{(n)}_{s + \delta_n}$ and thus $(A_s)_{-(s+\delta_n + 2\epsilon_n)}\cap \pc^{(n)} \subseteq (\hat{A}^{(n)}_{\delta_n + s})_{-(s+\delta_n + 2\epsilon_n)}$. Summarizing, we have shown that there is a point $q$ in $(\hat{A}^{(n)}_{\delta_n + s})_{-(s+\delta_n + 2\epsilon_n)}$ that is not in $\hat{A}^{(n)}$. By the change of variables $\tilde{s} := s + \delta_n + \epsilon_n$, it follows that $(\hat{A}^{(n)}_{\tilde{s}-\epsilon_n})_{-(\tilde{s} + \epsilon_n)} \setminus \hat{A}^{(n)} \neq \emptyset$ for all $\tilde{s} > r + \delta_n + \epsilon_n$, which implies $\rnv{A}{n} \leq r +\delta_n + \epsilon_n$. Sending $n\to\infty$ yields $\lim_{n\to \infty}\rnv{A}{n} \leq r$, and since $r \in (\rconv(A),\infty)$ was chosen freely, $\lim_{n\to \infty}\rnv{A}{n} \leq \rconv(A)$. This result, along with~\eqref{eqn:rconv_bound}, gives the convergence in~\eqref{eqn:rconv_lim}.
\end{proof}

\paragraph*{Justification for Equations~\eqref{eqn:example_reach} and~\eqref{eqn:example_beta_reach}.}
The reach of $A$ is equal to the inverse of the curvature of $f$ at $x=0$ \cite[Theorem~3.4]{aamari2019}. Equation~\eqref{eqn:example_reach} holds since the curvature of $f$ at $x=0$ is $f''(0) = 2h'(0)$.

Now we show~\eqref{eqn:example_beta_reach}.
Without loss of generality, suppose $h(0) = f(0) = 0$. For $x$ in a neighbourhood of 0, consider the symmetric points $a_1 = (x,f(x))$ and $a_2 = (-x,f(x))$ in $A$, and remark that
\begin{equation}\label{eqn:remark_for_justin}
g_{\norm{a_2-a_1}}\circ \delta_A\left(\frac{a_1 + a_2}2\right) = \frac{x^2}{2f(x)} + \frac{f(x)}2,
\end{equation}
since $\delta_A\left(\frac{a_1 + a_2}2\right)=f(x)$.
We claim without proof that for all $\tilde{a_1},\tilde{a_2}\in A$ satisfying $\delta_A\left(\frac{\tilde{a_1}+\tilde{a_2}}2\right) \geq f(x)$, it holds that $g_{\norm{\tilde{a_2}-\tilde{a_1}}}\circ \delta_A\left(\frac{\tilde{a_1} + \tilde{a_2}}2\right) \geq g_{\norm{a_2-a_1}}\circ \delta_A\left(\frac{a_1 + a_2}2\right)$.
In the language of the $\beta$-reach, this is equivalent to
\begin{equation}\label{eqn:equality_for_fustin}
\reach_{f(x)}(A) = g_{\norm{a_2-a_1}}\circ \delta_A\left(\frac{a_1 + a_2}2\right).
\end{equation}
There is a $\delta > 0$ such that $h$ has an inverse on $[0,\delta]$, and so for $\beta\in [0,\delta]$, choose $x$ such that $f(x)=\beta$. Remark that
$$x^2 = h^{-1}(\beta) = \frac\beta{h'(0)} - \frac{\beta^2 h''(0)}{2h'(0)^3} + \littleo(\beta^2),$$
by Taylor's theorem. Plugging into~\eqref{eqn:remark_for_justin} and applying~\eqref{eqn:equality_for_fustin}, one obtains~\eqref{eqn:example_beta_reach}.

\vspace{30pt}
%\newpage
\bibliographystyle{apalike}
\bibliography{biblio.bib} 
\end{document}